\documentclass[reqno,11pt]{amsart}
\DeclareMathOperator*{\esssup}{\mathrm{ess\,sup}}
\DeclareMathOperator*{\essinf}{\mathrm{ess\,inf}}

\usepackage{amssymb,amsthm}
\usepackage{amsmath}
\usepackage{amsfonts}
\usepackage{dsfont}

\usepackage[a4paper,  margin=3.34cm]{geometry}

\usepackage[active]{srcltx}
\makeatletter\@addtoreset{equation}{section}\makeatother

\newtheorem{theorem}{Theorem}[section]
\newtheorem{corollary}[theorem]{Corollary}
\newtheorem{lemma}[theorem]{Lemma}

\newtheorem{proposition}[theorem]{Proposition}
\newtheorem{assumption}[theorem]{Assumption}
\newtheorem{definition}[theorem]{Definition}
\newtheorem{remark}[theorem]{Remark}
\numberwithin{equation}{section}

\title[Evolution of infinite populations of immigrants]{Evolution of infinite populations of immigrants: micro- and mesoscopic description}

\author{ Yuri  Kozitsky}

\address{Instytut Matematyki, Uniwersytet Marii Curie-Sk{\l}odowskiej, 20-031 Lublin, Poland}
\email{jkozi@hektor.umcs.lublin.pl}

\begin{document}


\begin{abstract}

A model is proposed of an infinite population of entities
immigrating to a noncompact habitat, in which the newcomers are
repelled by the already existing population. The evolution of such a
population is described at micro- and mesoscopic levels. The
microscopic states are probability measures on the corresponding
configuration space. States of populations without interactions are
Poisson measures, fully characterized by their densities. The
evolution of micro-states is Markovian and obtained from the
Kolmogorov equation with the use of correlation functions. The
mesoscopic description is made by a kinetic equation for the
densities. We show that the micro-states are approximated by the
Poissonian states characterized by the densities obtained from the
kinetic equation. Both micro- and mesoscopic descriptions are
performed and their interrelations are analyzed, that includes also
discussing the problem of the appearance of a spatial diversity in
such populations.

\end{abstract}

\maketitle

\section{Introduction}

It is a common viewpoint that the observed spatial diversity of a
complex system appears due to the interplay between environmental
heterogeneities and interactions between its constituents. This
equally refers to systems studied in biology \cite{Kondo}, chemistry
\cite{Pini}, physics and other sciences \cite{Kicheva}. Despite a
substantial progress achieved during the last decades it is still a
challenging problem of applied mathematics to elaborate adequate
models and tools for studying these phenomena, see
\cite{Bellomo,Kicheva,Roth}.

In this work, we introduce and analyze a seemingly simple model in
which point entities arrive at random in a noncompact habitat, for
convenience chosen to be $\mathds{R}^d$, $d\geq 1$. The already
existing entities may repel the newcomers. No other actions -- like
birth, death (departure), jumps -- are taken into account. The
entities can be immigrants, molecules, ions, micro-organisms, etc.
The simplicity of the model allows us to analyze various aspects of
its description at micro- and mesoscopic levels and their
interconnections. At the same time, the model is rich enough to
capture the main peculiarities of the dynamics of such populations,
and some of the aspects of its theory developed here are quite
demanding. The model can also be used as a part of more involved
models, for which our present analysis can be used as a starting
point in their study.

If the arrival intensity $b$ (probability density per time) depends
only on the location $x\in \mathds{R}^d$, at the microscopic level the dynamics of the
population can be described as a spatio-temporal Poisson process,
cf. \cite{Kingman}, for which the probability $P_{t,\Lambda}(N)$ of
having $N$ entities at time $t$ in a compact set $\Lambda \subset
\mathds{R}^d$ is given by the Poisson law\footnote{The described
population gets instantly infinite if $\int_{\mathds{R}^d} b(x) d
x=+\infty$; i.e., if $b$ is not integrable, as it is in the
homogeneous case with constant $b$.}
\begin{equation}
  \label{Poisson}
P_{t,\Lambda}(N) = \frac{1}{N!}\left(t\int_\Lambda b(x) dx\right)^N
\exp\left(-t\int_\Lambda b(x) dx\right).
\end{equation}
In this case, the particle density is $\varrho_t(x) = t b(x)$ and
the only source of the diversity is the $x$-dependence of $b$. As is
typical for infinite systems of this kind, their microscopic
description turns into a hard mathematical problem whenever one
wants to include interactions. A more practical approach is to
describe such systems by using aggregate characteristics, like
particle density. In this case, however, one loses the possibility
to directly include interactions as no individual entities
participate in the description. A usual bypass here -- borrowed from
statistical physics where it is known as the mean-field approach --
is to make the model parameters state-dependent. This yields so
called phenomenological (or heuristic) models, that often have no
microscopic analogies and thus can provide rather superficial
description of the corresponding phenomena. At this level, the
evolution of our model is described by the following equation
\begin{equation}
  \label{Z1}
  \frac{d}{dt} \varrho_t (x) = b(x, \varrho_t) = b(x) \exp\left(
 - \int_{\mathds{R}^d} K(x,y) \varrho_t(y) dy \right),
\end{equation}
where $K(x,y)$ is a positive kernel, and hence the
interaction is repulsive. By setting $K(x,y) =
\phi(x-y) = \phi(y-x)$ one obtains a translation invariant version.
It is possible to show, see Theorem \ref{3tm} below, that this
equation has a unique solution $\varrho_t\in
L^\infty(\mathds{R}^d)$ if $\phi$ is bounded and integrable and $b$
is bounded. In the homogeneous case $b(x) \equiv b_*$, this solution
(with the zero initial condition) is also homogeneous and can be obtained
explicitly. It is
\begin{equation}
  \label{10}
\varrho_t (x) \equiv \frac{1}{\langle \phi \rangle}\log \bigg{(} 1 +
b_* \langle \phi \rangle t  \bigg{)}, \qquad \langle \phi \rangle:=
\int_{\mathds{R}^d} \phi (x ) d x.
\end{equation}
To illustrate that this homogeneity of the solution may be unstable
to arbitrarily small perturbations of the homogeneity of $b$ let us
consider the following version of (\ref{Z1}). Assume that there
exist two (compact) patches, $A$ and $B$, such that $b|_A = b_A$,
$b|_B = b_A$ and $b(x)=0$ whenever $x$ is outside of $A\cup B$.
Assume also that $K(x,y)=K(y,x)=1$ for $x\in A$ and $y\in B$,
$K(x,y) =\alpha\geq 0$ for $x,y\in A$ and for $x,y\in B$, and
$K(x,y)=0$ otherwise. Then consider the problem in (\ref{Z1}) with
the zero initial condition. Clearly, in this case one gets
$\varrho_t(x) =0$ for all $t\geq 0$ and $x$ outside of $A\cup B$.
Thus, one can consider the patches only. The restrictions
$\varrho_t|_A=: \varrho^A_t$ and $\varrho_t|_B=: \varrho^B_t$
satisfy the system of equations
\begin{equation}
  \label{Z2}
  \frac{d}{dt} \varrho^A_t = b_A \exp\left(-\alpha \varrho_t^A - \varrho_t^B\right),
  \qquad  \frac{d}{dt} \varrho^B_t = b_B \exp\left( -
  \varrho_t^A-\alpha \varrho_t^B\right).
\end{equation}
The case of $\alpha < 1$ (resp. $\alpha >1$) corresponds to the
situation in which the repulsion from the entities in the other
patch is stronger (resp. weaker) than that of the entities in the
same patch. The solution of this system clearly exists. However, it
can explicitly be  obtained only in particular cases. In general, one
might integrate (\ref{Z2}) and then analyze the solution. Recalling
that (\ref{Z2}) is subject to the zero initial condition, for
$\alpha\neq 1$ we get
\begin{equation}
  \label{Z3}
\exp\left((\alpha -1) \varrho^A_t \right) - 1 =\frac{b_A}{b_B}
\bigg{(} \exp\left((\alpha -1) \varrho^B_t \right) - 1\bigg{)}
\end{equation}
For $\alpha=1$, after some calculations we obtain an explicit
solution
\begin{equation}
  \label{Z4}
  \varrho^A_t = \frac{b_A}{b_A + b_B} \log\left(1 + (b_A + b_B)t
  \right), \quad  \varrho^B_t = \frac{b_B}{b_A + b_B} \log\left(1 + (b_A + b_B)t
  \right).
\end{equation}
In the latter case, both densities increase with time ad infinitum
logarithmically, and $\varrho^A_t \approx \varrho^B_t$ (at all $t$)
whenever $b_A \approx b_B$. For $\alpha >1$ and $b_A\neq b_B$, by
(\ref{Z3}) we conclude that both $\varrho^A_t$ and $\varrho^B_t$
increase ad infinitum and $\varrho^A_t \approx \varrho^B_t$ whenever
$b_A \approx b_B$. That is, in both these cases no clear difference
between the patches appears if the intensities $b_A$ and $b_B$ are close to each other.
In the homogeneous case $b_A = b_B =
b_*$, for all $\alpha$ we get
\begin{equation*}
\varrho_t^A = \varrho_t^B = \varrho_t = \frac{1}{1+\alpha}
\log\left(1 + (1+\alpha) b_* t \right),
\end{equation*}
that resembles (\ref{10}). For $\alpha <1$, however, the
equality between the patches in the homogeneous case gets unstable.
Assuming $b_A<b_B$, by
(\ref{Z3}) we conclude that $\varrho_t^B$ should increase ad
infinitum whereas $\varrho_t^A$ tend to $\varrho_\infty^A = [\log
b_B - \log (b_B - b_A) ]/(1-\alpha)$. This effect of increasingly
diverse patches persists regardless how small is $b_B - b_A$. Having
this effect in mind, one might be interesting in developing a more
or less complete theory of (\ref{Z1}), in particular, in realizing
which properties of $b$ and $K$ are responsible for pattern
formation. In part, this is done in Theorem \ref{3tm} below, where
we show the existence of a unique classical solution and derive some
information on its properties. Clearly, in studying (\ref{Z1})
numerical methods can be essentially helpful. We plan to realize
this in a separate work.

As mentioned above, a theory based on kinetic equations like in
(\ref{Z1}) is pretty rough and can only be considered as the first
step in the study of the corresponding phenomena at which spatial
correlations are taken into account in a mean-field like way. To
make the next step one should (in one or another way) obtain the
`true' and `complete' system -- hierarchy -- of equations, that
include correlation functions of second, third, and higher orders.
Then one might think of a decoupling of this hierarchy, see, e.g.,
\cite{Baker,Moment}. A yet more sophisticated issue is whether a
solution of this `true' hierarchy gives correlation functions
corresponding to a microscopic state. If yes, then the evolution
obtained by solving the hierarchy can `represent' -- in a certain
sense -- the evolution of the micro-states and thus yield a
microscopic description of the considered phenomenon. In this work,
for our model -- introduced in Section
  \ref{2S} -- we give partial answers to these questions. This
includes the following:
\begin{itemize}
  \item[(i)] Deriving a hierarchy of evolution equations for correlation
 functions from the Kolmogorov equation describing the microscopic
 evolution of observables (subsection \ref{2.3ss} and Appendix).
\item[(ii)] Proving the existence and uniqueness of solutions of the evolution equation for correlation
functions that yields the evolution $k_0\to k_t$; then proving that
each $k_t$ is the correlation function of a unique micro-state
$\mu_t$ (Theorem \ref{1tm}), that yields the evolution of states
$\mu_0\to \mu_t$.
\item[(iii)] Under an additional condition on the repulsion,
proving (Theorem \ref{2tm}) that the expected number of entities
contained in a compact $\Lambda$ increases in time at most
logarithmically, cf. (\ref{10}).
\item[(iv)] Proving the existence and uniqueness of solutions of
(\ref{Z1}) and describing some of its properties (Theorem
\ref{3tm}).
\item[(v)] Deriving (by a scaling procedure) a mesoscopic evolution
equation -- that coincides with the kinetic equation (\ref{Z1}) -- from
the Kolmogorov equation describing the evolution of micro-states of
the model. Proving that the mesoscopic limit of the state $\mu_t$ is
the Poisson state characterized by the density that solves
(\ref{Z1}) (Theorem \ref{4tm}).
\end{itemize}
The paper is organized as follows. In Section \ref{2S}, we briefly
introduce the mathematical framework in which we then build up the
theory. Then we introduce the model by defining the
generator of its Markov evolution. Thereafter, we discuss in detail
how to describe a weak evolution of states of infinite particle
systems, including the model considered in this work. The main idea
of this is to use correlation functions the evolution of which is
described by an evolution equation derived from the Kolmogorov
equation in Appendix. The latter equation has a hierarchical form
and is considered in an ascending scale of Banach space chosen in
such a way that they contain the correlation functions of the
so-called sub-Poissonian states. To relate the microscopic theory
based on the Kolmogorov equation to the one that uses kinetic
equations we introduce the notion of Poisson approximability of the
micro-states (Definition \ref{Poidf}) based on the passage from the
micro- to the meso-scale. Finally, we formulate our results in
Theorems \ref{1tm} -- \ref{4tm} and Corollary \ref{M1co}. In
subsection \ref{CommentsSS} we comment these results and give
additional information on their meaning and relevance. In Section
\ref{ESs}, we give the proof of Theorems \ref{1tm} and \ref{2tm}.
The proof of the latter theorem is rather simple and based on the
results of Theorem \ref{1tm}. Its proof, however, is the
most involved part of this work. In is based on a
combination of a number of methods of studying evolution equations
in scales of Banach spaces, including the theory of stochastic
semigroups in $AL$-spaces. Section \ref{MESS} contains the proof of
Theorems \ref{3tm} and \ref{4tm} describing the mesoscopic evolution
and its connection to the microscopic evolution. In the concluding
part of the paper we placed Appendix containing technicalities.

\section{The Model and the Results}
\label{2S}

Here we give the description of the model preceded by a short
presentation of the preliminaries (see \cite{KK1,KK2} for more) and
then briefly outline the main aspects of the theory. Thereafter, we
formulate the results.

\subsection{Preliminaries}

The microscopic description of the dynamics of the model which we
introduce and study in this work is conducted as a Markov evolution
of states of an infinite population of point entities dwelling
$\mathds{R}^d$, cf. \cite{FKKK,KK1,KK2}. The phase space of the
population is the set $\Gamma$ of all locally finite subsets
$\gamma\subset \mathds{R}^d$ -- \emph{configurations}, equipped with
a standard (vague) topology and the corresponding Borel
$\sigma$-field $\mathcal{B}(\Gamma)$. This makes $(\Gamma,
\mathcal{B}(\Gamma))$ a standard Borel space that allows one to
consider probability measures on $\Gamma$ as states of the system.
In the sequel, $\mathcal{P}(\Gamma)$ will stand for the set of all
such measures. Along with states $\mu\in\mathcal{P}(\Gamma)$ one
employs \emph{observables} -- appropriate measurable functions
$F:\Gamma \to \mathds{R}$. The mentioned local finiteness means that
the observable $\Gamma \ni \gamma \mapsto |\gamma \cap \Lambda|=:
N_\Lambda (\gamma)$ takes finite values only. Here $\Lambda$ is a
compact subset of $\mathds{R}^d$ and $|\cdot |$ stands for
cardinality. Among the states one can distinguish those
characterized by the corresponding densities -- locally intergrable
functions $\varrho_\mu:\mathds{R}^d \to \mathds{R}_{+}:=[0,+\infty)$
such that
\begin{equation}
  \label{0}
  \mu (N_\Lambda) = \int_{\Lambda} \varrho_\mu (x) d x,
\end{equation}
where $dx$ is Lebesgue's measure and the notation $\mu(F):=\int F d
\mu$ is used for appropriate measures and observables. Note that $
\mu (N_\Lambda)$ is the (expected) number of `particles' contained
in $\Lambda$ in state $\mu$. Each Poissonian measure $\pi_\varrho$,
see \cite{Kingman}, is completely characterized by its density
$\varrho$. In the homogeneous case, $\varrho(x)\equiv \varkappa>0$,
and thus $\pi_\varkappa$ is invariant with respect to the
translations of $\mathds{R}^d$.

The dynamics of a given system can be described as the evolution of
observables by solving the Kolmogorov equation
\begin{equation}
  \label{1}
  \frac{d}{dt} F_t = L F_t, \qquad F_t|_{t=0}=F_0.
\end{equation}
The evolution of states $\mathcal{P}(\Gamma) \ni \mu_0 \to \mu_t \in
\mathcal{P}(\Gamma)$ is then usually obtained (in the weak sense) by
means of the rule
\begin{equation}
  \label{2}
  \mu_t (F_0) = \mu_0 (F_t),
\end{equation}
where the function $t \mapsto F_t$ is obtained from (\ref{1}) with
$F_0$ running through a sufficiently large (measure-defining) set of
observables. As such a set one can take $\mathcal{F}=\{F^\theta :
\theta \in \varTheta\}$, where
\begin{equation}
  \label{3}
  F^\theta (\gamma) = \prod_{x\in \gamma} ( 1 + \theta (x)), \qquad
  \gamma \in \Gamma,
\end{equation}
and $\varTheta$ stands for the set of compactly supported continuous
functions $\theta:\mathds{R}^d \to (-1,0]$.

The model studied in this work is a particular case of a more
general one describing a system of point entities arriving in
(immigrating) and departing from (emigrating) $\mathds{R}^d$. Such
systems are described by the Kolmogorov equations with
\begin{eqnarray}
  \label{4}
 (LF)(\gamma) & = & \int_{\mathds{R}^d} c_{+} (x, \gamma) \left[F(\gamma \cup x) - F(\gamma)
 \right]dx \\[.2cm] \nonumber & + & \sum_{x\in \gamma} c_{-} (x ,\gamma \setminus x) \left[F(\gamma \setminus x) - F(\gamma)
 \right].
\end{eqnarray}
Here and in the sequel, in the expressions like $\gamma\cup x$ we
treat $x$ as a singleton configuration. If the immigration and
emigration rates $c_{\pm}$ are state-independent, the evolution of
states $\mu_0 \to \mu_t$ can be constructed explicitly, see
\cite[Sect. 2.3]{KK2}. In this case: (a) $\varrho_{\mu_t}(x) =
\varrho_{\mu_0}(x) + t c_{+}(x)$ if $c_{-}(x) \equiv 0$; (b)
$\varrho_{\mu_t}(x) \leq \varrho_{\infty}(x)$ for some
$\varrho_{\infty}(x)<\infty$ whenever $c_{-} (x) >0$ at this $x$. By
(\ref{0}) (b) implies
\begin{equation}
  \label{5}
  \mu_t (N_\Lambda) \leq C_{\Lambda},
\end{equation}
holding for all $t>0$ and compact $\Lambda \subset \mathds{R}^d$
such that $c_{-}(x) \geq \sigma_\Lambda >0$ for all $x\in \Lambda$.
In particular, if the condition $c_{-}(x) \geq \sigma >0$ is
satisfied globally (for all $x$), then (\ref{5}) holds also
globally, i.e., for all compact $\Lambda$. The main question studied
in \cite{KK2} was whether a local competition alone can yield the
global boundedness as in (\ref{5}). In this case, $c_{-}$ has the
form
\begin{equation*}
  c_{-} (x, \gamma) = \sum_{y\in \gamma} a_{-} (x,y),
\end{equation*}
with an appropriate competition kernel $a_{-}(x,y)\geq 0$. Here,
however, the particles in the system interact with each other and
the proof of the existence of the evolution of states turns into an
essential problem. Nevertheless, in \cite{KK2} the  weak evolution
$\mu_0\to \mu_t$ was constructed for which it was shown that
(\ref{5}) holds for each compact $\Lambda$ whenever there exists a
box $\Delta \subset \mathds{R}^d$ such that its disjoint translates
$\Delta_l$, $l\in \mathds{Z}^d$, cover $\mathds{R}^d$ and $a_{-}
(x,y) \geq a_0>0$ for $x,y$ running through each of $\Delta_l$. In
the present article, we study the particular case of (\ref{4}) in
which emigration is absent, i.e., $c_{-}(x,\gamma) \equiv 0$, and
the immigration term is taken in the form, cf. (\ref{Z1}),
\begin{equation}
  \label{7b}
  c_{+} (x, \gamma) = b(x) \exp \left( - \sum_{y\in \gamma} \phi(x-y)
  \right),
\end{equation}
where $\phi:\mathds{R}^d\to \mathds{R}_{+}$ is the \emph{repulsion}
potential that is assumed to be integrable and bounded. According to
(\ref{7b}) the already existing population decreases the immigration
rate as compared to the free case $\phi(x)\equiv 0$. As mentioned above,
in this free case one has
\begin{equation*}
\mu_t(N_\Lambda) = \mu_0 (N_\Lambda) + t \int_{\Lambda} b(x) dx.
\end{equation*}
Then one of the questions of the theory is how the repulsion contained in (\ref{7b})
can attenuate the growth of $\mu_t(N_\Lambda)$ in time.  According
to the assumed boundedness of $\phi$, which, in particular, excludes
a hard-core repulsion, we have that $c_{+} (x, \gamma)>0$ for all
those $x$ where $b(x)>0$. Thus, one cannot expect that
$\mu_t(N_\Lambda)$ be bounded in time since any kind of emigration
is absent in the model.

\subsection{The model}
In this article, we introduce and study the model described by the
Kolmogorov equation with the following operator, cf. (\ref{4})
 \begin{equation}
   \label{17}
  (L F)(\gamma) = \int_{\mathds{R}^d} b(x) \exp\left(- \sum_{y\in
  \gamma} \phi(x-y)
  \right) \left[ F(\gamma \cup x) - F(\gamma)\right] d x.
 \end{equation}
\begin{assumption}
  \label{ass}
In general, we assume that: (a) $b$ and $\phi$ are nonnegative,
measurable and bounded from above by $\bar{b}$ and $\bar{\phi}$,
respectively; (b) $\phi$ is integrable and its $L^1$-norm is denoted
by $\langle \phi \rangle$, see (\ref{10}). Additionally, we assume:
(c) there exist $r>0$ and $\phi_*>0$ such that
\begin{equation*}
  \phi(x) \geq \phi_*, \qquad {\rm whenever} \ |x|\leq r.
\end{equation*}
\end{assumption}
Notably, we do not assume that $b$ is integrable, which means that
we allow the particle system described by (\ref{17}) with such $b$
be instantly infinite, even if it is initially empty.

\subsection{The weak evolution of states}
\label{2.2ss} For a given $\mu\in \mathcal{P}(\Gamma)$, its {\it
Bogoliubov} (generating) functional is $B_\mu (\theta) =
\mu(F^\theta)$ with $F^\theta$ defined in (\ref{3}). It thus is
considered as a map from $\varTheta$ to $\mathds{R}_{+}$. For a
Poisson measure $\pi_\varrho$  characterized by density $\varrho$,
it is
\begin{equation}
  \label{I1}
B_{\pi_\varrho} (\theta) = \exp\left( \int_{\mathbb{R}^d}\varrho(x)
\theta (x) d x \right).
\end{equation}
In our
consideration, Poisson measures play the role of reference states.
In view of this, we restrict our attention to the subset
$\mathcal{P}_{\rm exp}(\Gamma) \subset \mathcal{P}(\Gamma)$
containing all those $\mu\in \mathcal{P}(\Gamma)$ for each of which
$B_\mu$ can be continued, as a function of $\theta$, to an
exponential type entire function on $L^1 (\mathbb{R}^d)$. It can be
shown that a given $\mu$ belongs to $\mathcal{P}_{\rm exp}(\Gamma)$
if and only if its functional $B_\mu$ can be written down in the
form
\begin{eqnarray}
  \label{I3}
B_\mu(\theta) = 1+ \sum_{n=1}^\infty
\frac{1}{n!}\int_{(\mathbb{R}^d)^n} k_\mu^{(n)} (x_1 , \dots , x_n)
\theta (x_1) \cdots \theta (x_n) d x_1 \cdots d x_n,
\end{eqnarray}
where $k_\mu^{(n)}$ is the $n$-th order correlation function of
$\mu$. It is a symmetric element of $L^\infty ((\mathbb{R}^d)^n)$
for which
\begin{equation}
\label{I4}
  \|k^{(n)}_\mu \|_{L^\infty
((\mathbb{R}^d)^n)} \leq \exp( \vartheta n), \qquad n\in
\mathbb{N}_0,
\end{equation}
with a certain $\vartheta \in \mathds{R}$. It is known, see Proposition \ref{Tobipn} below,
that $k^{(n)}_\mu (x_1, \dots , x_n) \geq 0$ for all $n$ and (Lebesgue-) almost all $x_1, \dots , x_n$. Then
(\ref{I4}) can be rewritten in the form
\begin{equation}
 \label{I4a}
 0 \leq k^{(n)} (x_1, \dots , x_n) \leq \exp( \vartheta n), \qquad n\in
\mathbb{N}_0.
\end{equation}
Note that
$k_{\pi_\varkappa}^{(n)} (x_1 , \dots , x_n)= \varkappa^n$. Note also that $k^{(0)}=1$ and $k^{(1)}_\mu$ is the
density of the system in state $\mu$. The set $\Gamma$
contains a subset, $\Gamma^{(0)}$, consisting of a single element -- the empty
configuration. If $\mu$ is such that $\mu(\Gamma^{(0)})=1$, then the system in this state
is empty. Its correlation function $k_{\mu}$ satisfies (\ref{I4}) with any $\vartheta$.
In this case, we will allow $\vartheta$ in (\ref{I4}) and (\ref{I4a}) be $-\infty$.
For $\mu\in \mathcal{P}_{\rm exp}(\Gamma)$ and a compact $\Lambda$, by (\ref{I4a}) and \cite[eq. (4.5)]{KK2} it follows that
\begin{equation}
 \label{I4c}
\forall m\in \mathds{N} \qquad  \mu(N_\Lambda^m) \leq \mu_{\pi_\varkappa} (N_\Lambda^m), \quad \varkappa = e^\vartheta,
\end{equation}
with $\vartheta$ satisfying (\ref{I4a}) for this $\mu$.
That is why the elements of
$\mathcal{P}_{\rm exp}(\Gamma)$ are called sub-Poissonian states. By (\ref{I4c}) one readily gets that
$\mu\in \mathcal{P}_{\rm exp}(\Gamma)$ also satisfies
\begin{equation}
\label{I4d} \forall \alpha >0 \qquad \mu\left(F^\alpha_\Lambda
\right) \leq \exp\left((e^\alpha -1) e^\vartheta |\Lambda| \right):=
C^\alpha_\Lambda (\mu), \quad F^\alpha_\Lambda (\gamma):= e^{\alpha
N_\Lambda(\gamma)},
\end{equation}
where $\vartheta$ is the same as in (\ref{I4c}) and $|\Lambda|$ stands for Lebesgue's measure of $\Lambda$.

Let $\Gamma^{(n)}$, $n\in \mathds{N}_0$, stand for the set of
$n$-point configurations. It is an element of $\mathcal{B}(\Gamma)$.
Then so is $\Gamma_0 = \cup_{n\geq 0} \Gamma^{(n)}$ -- the set of
all finite configurations. In the sequel, we will use also the
measurable space $(\Gamma_0, \mathcal{B}(\Gamma_0))$ where
$\mathcal{B}(\Gamma_0)$ is a sub-field of $\mathcal{B}(\Gamma)$
containing all measurable subsets of $\Gamma_0$. It can be shown,
see \cite{FKKK}, that a function $G:\Gamma_0 \to \mathds{R}$ is
$\mathcal{B}(\Gamma)/\mathcal{B}(\mathds{R} )$-measurable if and
only if, for each $n\in \mathds{N}$, there exists a symmetric Borel
function $G^{(n)}: (\mathds{R}^{d})^{n} \to \mathds{R}$ such that
\begin{equation}
 \label{7}
 G(\eta) = G^{(n)} ( x_1, \dots , x_{n}), \quad {\rm for} \ \eta = \{ x_1, \dots , x_{n}\}.
\end{equation}
\begin{definition}
  \label{Gdef}
A measurable function $G:\Gamma_0 \to \mathds{R}$  is said to have
bounded support if: (a) there exists a compact $\Lambda\subset
\mathds{R}^d$ such that $G(\eta) = 0$ whenever $\eta\cap
(\mathds{R}^d \setminus \Lambda)\neq \varnothing$; (b) there exists
$N\in \mathds{N}_0$ such that $G(\eta)=0$ whenever $|\eta|
>N$.  By $\Lambda(G)$ and $N(G)$ we denote the
smallest $\Lambda$ and $N$ with the properties just mentioned. By
$B_{\rm bs}(\Gamma_0)$ we denote the set of all bounded functions of bounded support.
\end{definition}
In our study, we use the following subset of $B_{\rm bs}(\Gamma_0)$
\begin{equation}
  \label{7z}
B^\star_{\rm bs}(\Gamma_0) = \{G\in B_{\rm bs}(\Gamma_0): (K
G)(\eta) \geq 0\}, \qquad (K G)(\eta) := \sum_{\xi\subset\eta}
G(\xi).
\end{equation}
Note that the cone $B^+_{\rm bs}(\Gamma_0) = \{G\in B_{\rm
bs}(\Gamma_0):  G(\eta) \geq 0\}$ is a proper subset of
$B^\star_{\rm bs}(\Gamma_0)$.

The Lebesgue-Poisson measure $\lambda$ on $(\Gamma_0,
\mathcal{B}(\Gamma_0))$ is defined by the following formula
\begin{eqnarray}
\label{8} \int_{\Gamma_0} G(\eta ) \lambda ( d \eta)  =
G(\varnothing) + \sum_{n=1}^\infty \frac{1}{n! }
\int_{(\mathds{R}^d)^{n}} G^{(n)} ( x_1, \dots , x_{n} ) d x_1
\cdots dx_{n},
\end{eqnarray}
holding for all appropriate $G:\Gamma_0\to \mathds{R}$, that
obviously includes $G\in B_{\rm bs}(\Gamma_0)$. Like in (\ref{7}),
we introduce $k_\mu : \Gamma_0 \to \mathds{R}$ such that
$k_\mu(\eta) = k^{(n)}_\mu (x_1, \dots , x_n)$ for $\eta = \{x_1,
\dots , x_n\}$, $n\in \mathds{N}$. We also set
$k_\mu(\varnothing)=1$. With the help of the measure introduced in
(\ref{8}), the expressions for $B_\mu$ in (\ref{I1}) and (\ref{I3})
can be combined into the following formulas
\begin{eqnarray}
  \label{1fa}
 B_\mu (\theta)& = & \int_{\Gamma_0} k_\mu(\eta) \prod_{x\in \eta} \theta (x) \lambda (d\eta)=: \int_{\Gamma_0} k_\mu(\eta) e(\theta; \eta) \lambda (d \eta)
 \\[.2cm]
 & = &  \int_{\Gamma} \prod_{x\in \gamma} (1+ \theta (x)) \mu (d \gamma) = \int_{\Gamma} F^\theta (\gamma) \mu(d
 \gamma), \nonumber
\end{eqnarray}
cf. (\ref{3}). Thereby, one can transform  the action of $L$ on $F$,
see (\ref{4}), to the action of $L^\Delta$ on $k_\mu$ according to
the rule
\begin{equation}
  \label{1g}
\int_{\Gamma}(L F_\theta) (\gamma) \mu(d \gamma) = \int_{\Gamma_0}
(L^\Delta k_\mu) (\eta) e(\theta;\eta)
 \lambda (d \eta).
\end{equation}
Correspondingly, along with the Kolmogorov equation (\ref{1}) one can consider
\begin{equation}
 \label{15}
 \frac{d}{dt} k_t = L^\Delta k_t , \qquad k_t|_{t=0} = k_{\mu_0},
\end{equation}
where $\mu_0 \in \mathcal{P}_{\rm exp} (\Gamma)$ is the initial state. Then the construction of
the evolution of states $\mu_0 \to \mu_t$ will be realized in the following steps:
\begin{itemize}
 \item[(i)] proving the existence of a unique solution $k_t$ of the Cauchy problem in (\ref{15});
 \item[(ii)] proving that this solution $k_t$ is the correlation function of a unique $\mu_t\in \mathcal{P}_{\rm exp} (\Gamma)$.
\end{itemize}
Upon realizing this program, one will be able to identify $\mu_t$ by
its values on the members of $\mathcal{F}$ -- a measure-defining
class -- computed according to the formula
\begin{equation}
 \label{16}
 \mu_t(F^\theta) = \int_{\Gamma_0} k_t(\eta) e(\theta;\eta) \lambda ( d \eta),
\end{equation}
see (\ref{1fa}). In realizing item (ii), we will use the following
statement, see \cite[Theorems 6.1 and 6.2 and Remark 6.3]{Tobi}.
\begin{proposition}
  \label{Tobipn}
A measurable function $k:\Gamma_0\to \mathds{R}$ is the correlation
function of a unique $\mu\in \mathcal{P}_{\rm exp}(\Gamma_0)$ if and
only if it satisfies: (a) $k(\varnothing)=1$; (b) the estimate in
(\ref{I4}) holds for some $\vartheta\in \mathds{R}$ and all $n\in
\mathds{N}$; (c) for each $G\in B^\star_{\rm bs}(\Gamma_0)$, see
(\ref{7z}), the following holds
\begin{equation}
  \label{19}
  \langle\!\langle G, k \rangle \!\rangle := \int_{\Gamma_0} G(\eta)
  k(\eta) \lambda (d \eta) \geq 0.
\end{equation}
\end{proposition}

\subsection{Evolution of correlation functions}\label{2.3ss}

Now we turn to defining the Cauchy problem (\ref{15}) in appropriate
Banach spaces. In view of (\ref{I4}), we set
\begin{equation}
  \label{20}
  \|k\|_\vartheta = \sup_{n\in \mathds{N}_0} \left( \exp\left(
  - \vartheta n\right) \|k^{(n)}\|_{L^\infty((\mathds{R}^d)^n)}\right) =
  \esssup_{\eta \in \Gamma_0}\left( \exp\left(
  - \vartheta |\eta|\right) |k(\eta)|\right),
\end{equation}
and then
\begin{equation*}
  \mathcal{K}_\vartheta = \{ k:\Gamma_0 \to \mathds{R}:
  \|k\|_\vartheta < \infty\}, \qquad \vartheta \in \mathds{R}.
\end{equation*}
By (\ref{20})  one clearly gets that
\begin{equation}
  \label{22}
  |k(\eta)| \leq \exp\left( \vartheta |\eta|\right) \|k\|_\vartheta,
  \qquad \eta \in \Gamma_0,
\end{equation}
and
\begin{equation}
  \label{23}
  \mathcal{K}_\vartheta \hookrightarrow \mathcal{K}_{\vartheta'},
  \qquad {\rm for} \ \ \vartheta'> \vartheta,
\end{equation}
that is, $\mathcal{K}_\vartheta$ is continuously embedded in
$\mathcal{K}_{\vartheta'}$. The latter allows one to employ the
whole ascending scale of Banach spaces
$\{\mathcal{K}_\vartheta\}_{\vartheta \in \mathds{R}}$.

In this paper, we extensively use the Minlos lemma, cf. \cite[eq.
(2.2)]{FKKK} and/or \cite[Appendix A]{Tobi}, according to which the
following
\begin{equation}
  \label{Minlos}
  \int_{\Gamma_0}\int_{\Gamma_0} G(\eta\cup\xi) H(\eta,\xi)
  \lambda(d \eta) \lambda (d \xi) = \int_{\Gamma_0} G(\eta)
  \left(\sum_{\xi \subset \eta} H(\xi, \eta\setminus \xi)\right) \lambda(d
  \eta),
\end{equation}
holds for appropriate $G,H:\Gamma_0\to \mathds{R}$. In Appendix, by
means of (\ref{Minlos}) and (\ref{1g}) we show that
\begin{gather}
  \label{24}
(L^\Delta k)(\eta)  =  \sum_{x\in \eta} b(x) e(\tau_x;\eta\setminus
x) \left( W_x
k\right)(\eta \setminus x), \\[.2cm] \nonumber (W_x k)(\eta):=
\int_{\Gamma_0} e(t_x;\xi) k(\eta \cup \xi) \lambda (d \xi), \\[.2cm] \nonumber
\tau_x (y) :=  e^{-\phi(x-y)}, \qquad t_x (y) := \tau_x(y) -1.
\end{gather}
By (\ref{22}) we then have
\begin{eqnarray}
  \label{25}
\left|(W_x k)(\eta) \right| & \leq & \|k\|_\vartheta
e^{\vartheta|\eta|} \int_{\Gamma_0} e(|t_x|;\xi) e^{\vartheta |\xi|}
\lambda ( d\xi) \\[.2cm] \nonumber & = &  \|k\|_\vartheta
e^{\vartheta|\eta|} \exp\left( e^\vartheta \int_{\mathds{R}^d}
\left( 1 - e^{-\phi(x-y)}\right)d y\right)  \\[.2cm] \nonumber & \leq &  \|k\|_\vartheta
e^{\vartheta|\eta|} \exp\left( \langle \phi \rangle e^\vartheta
\right),
\end{eqnarray}
where we have used the assumed properties of $\phi$, see (\ref{10}).
Now we apply (\ref{25}) in (\ref{24}) and get
\begin{equation}
  \label{26}
\left| (L^\Delta k)(\eta) \right| \leq |\eta|e^{\vartheta|\eta|}
\bar{b}\|k\|_\vartheta \exp\left(\langle \phi \rangle e^\vartheta -
\vartheta \right),
\end{equation}
that holds for each $\vartheta \in \mathds{R}$. Set
\begin{equation}
  \label{27}
  \mathcal{D}_\vartheta = \left\{k\in \mathcal{K}_\vartheta: \exists C_k >0 \  |k(\eta)| \leq \frac{C_k  e^{\vartheta
  |\eta|}}{1 +|\eta|}\right\}.
\end{equation}
Similarly as in obtaining (\ref{26}) one shows that $L^\Delta k \in
\mathcal{K}_\vartheta$ for each $k \in \mathcal{D}_\vartheta$. Thus,
we define in $\mathcal{K}_\vartheta$ the (unbounded) linear operator
$L^\Delta_\vartheta := (L^\Delta, \mathcal{D}_\vartheta)$. By
(\ref{22}) and (\ref{27}) we readily obtain that
\begin{equation}
  \label{28}
 \forall \vartheta_0 < \vartheta \qquad  \mathcal{K}_{\vartheta_0}
 \subset \mathcal{D}_\vartheta.
\end{equation}
Moreover, by employing in (\ref{26}) the estimate
$|\eta|e^{-\upsilon |\eta|} \leq (\upsilon e)^{-1}$, $\upsilon >0$,
we conclude that, for each $\vartheta'>\vartheta$, cf. (\ref{23}),
$L^\Delta$ can be defined as a bounded linear operator -- denoted by
$L^\Delta_{\vartheta'\vartheta}$ -- acting from
$\mathcal{K}_\vartheta$ to $\mathcal{K}_{\vartheta'}$, the operator
norm of which satisfies
\begin{equation}
  \label{29}
 \|L^\Delta_{\vartheta'\vartheta}\| \leq \frac{\beta
 (\vartheta)e^{-\vartheta}}{e(\vartheta'-\vartheta)}, \qquad \beta (\vartheta) :=
 \bar{b} \exp\left(\langle \phi \rangle e^\vartheta \right).
\end{equation}
In the sequel, we will consider two types of linear operators
defined by (\ref{24}): (a) unbounded operators $L^\Delta_\vartheta:
\mathcal{D}_\vartheta \to \mathcal{K}_\vartheta$, $\vartheta\in
\mathds{R}$; (b) bounded operators $L^\Delta_{\vartheta'\vartheta}:
\mathcal{K}_{\vartheta}\to \mathcal{K}_{\vartheta'}$ with
$\vartheta'>\vartheta$. These operators satisfy
\begin{equation}
  \label{30}
 \forall \vartheta'> \vartheta \  \ \forall k\in \mathcal{K}_\vartheta \qquad
 L^\Delta_{\vartheta'\vartheta} k = L^\Delta_{\vartheta'} k.
\end{equation}
Now we fix some $\vartheta\in\mathds{R}$ and consider the following
Cauchy problem in $\mathcal{K}_\vartheta$
\begin{equation}
  \label{31}
  \frac{d}{dt} k_t = L^\Delta_\vartheta k_t, \qquad k_t|_{t=0}=k_0.
\end{equation}
\begin{definition}
  \label{1df}
By a (classical) solution of the problem in (\ref{31}) on a time
interval $[0,T)$ we understand a continuous map $[0,T)\ni t \mapsto
k_t  \in \mathcal{D}_\vartheta\subset \mathcal{K}_\vartheta$, which
is continuously differentiable on $(0,T)$ and such that both
equalities in (\ref{31}) are satisfied. We say that such a solution
is global (in time) if $T=+\infty$.
\end{definition}
Then item (i) of the program mentioned above assumes proving the
existence of such solutions. Note, however, that a priori $k_t$ that
solves (\ref{31}) need not be the correlation function of any state,
and hence has no direct relation to the evolution of the considered
system. To realize item (ii) that establishes such a relation we
show that this solution has the property $k_t \in \mathcal{K}^\star$
where
\begin{equation}
  \label{32}
\mathcal{K}^\star:= \bigcup_{\vartheta\in \mathds{R}}
\mathcal{K}^\star_\vartheta, \quad  \ \mathcal{K}^\star_\vartheta =
\{ k \in \mathcal{K}_\vartheta: k(\varnothing)=1 \ {\rm and} \
\forall G \in B^\star_{\rm bs}(\Gamma_0)
  \ \
  \langle \! \langle G,k\rangle \! \rangle \geq 0\},
\end{equation}
see (\ref{19}).

\subsection{The mesoscopic description}
\label{Mesoss}

It is believed that the description of an infinite interacting
particle system by means of kinetic equations is in a sense
equivalent to considering it at a more coarse-grained (mesoscopic)
spatial scale, see \cite[Chapter 8]{Jacek} and \cite{Presutti}.
Typically, passing from the micro- to the mesoscopic levels is made
with the help od a scale parameter �$\varepsilon\in(0, 1]$ in such a
way that �$\varepsilon = 1$ corresponds to the micro-level, whereas
the limit $\varepsilon \to 0$� yields the description in which the
corpuscular structure disappears and the system turns into a medium
characterized by density $\varrho$. In this limit, instead of
interactions one deals with state-dependent external forces, that is
typical to the mean-field approach of statistical physics. The
evolution $\varrho_0 \to \varrho_t$ is obtained from a kinetic
equation. It approximates the evolution of states as it may be seen
from the mesoscopic level provided these states exist. As the
Poissonian state $\pi_\varrho$ is completely characterized by its
density, cf. (\ref{I1}) (note that $k_{\pi_\varrho}(\eta)=
\prod_{x\in \eta} \varrho(x)$), the restriction to considering
densities only can be interpreted as approximating the states
$\mu_t$ by the corresponding Poissonian states $\pi_{\varrho_t}$. In
view of this, we introduce the following notion, cf.  \cite[p.
70]{FKKK}.
\begin{definition}
  \label{Poidf}
A state $\mu\in \mathcal{P}_{\rm exp} (\Gamma)$ is said to be
Poisson-approximable if: (i) there exist $\vartheta\in \mathds{R}$
and a bounded measurable $\varrho:\mathds{R}^d \to [0,+\infty)$ such
that both $k_\mu$ and $k_{\pi_\varrho}$ lie in the same
$\mathcal{K}_{\vartheta}$ ; (ii) for each $\varepsilon \in (0,1]$,
there exists $q^\varepsilon\in \mathcal{K}_{\vartheta}$ such that
$q^1 = k_\mu$ and $\|q^\varepsilon - k_{\pi_\varrho}\|_\vartheta \to
0$ as $\varepsilon \to 0^+$.
\end{definition}
Our aim is to show that the evolution of states $\mu_0\to \mu_t$
discussed above in subsection \ref{2.2ss} preserves the property
just defined relative to the time dependent density $\varrho_t$
obtained as a solution of the kinetic equation in (\ref{8z}),
understood similarly as in Definition \ref{1df}.

\subsection{The results}
First we establish the existence of the evolution of states as
discussed in subsections \ref{2.2ss} and \ref{2.3ss} and describe
some of its properties. Then we turn to the mesoscopic scale.

\subsubsection{The evolution of states}
\begin{theorem}
  \label{1tm}
Let the model satisfy items (a) and (b) of Assumption \ref{ass} and
let the initial state $\mu_0\in \mathcal{P}_{\rm exp}(\Gamma)$ and
$\vartheta_0 \in \mathds{R}$ be such that $k_{\mu_0}\in
\mathcal{K}_{\vartheta_0}$. Set $\vartheta_t = \log (e^{\vartheta_0}
+ \bar{b}t)$, $t\geq 0$. Then there exists the unique map
$[0,+\infty)\ni t \mapsto k_t \in \mathcal{K}^\star$, see
(\ref{32}), such that $k_0 = k_{\mu_0}$ and the following holds:
\begin{itemize}
  \item[{\it (i)}] for all $t>0$,
  \[
0 \leq k_t (\eta) \leq \exp\left( \vartheta_t |\eta|\right), \qquad
\eta \in \Gamma_0,
  \]
  and hence $k_t \in \mathcal{K}^\star_{\vartheta_t}$;
\item[{\it (ii)}] for each $T>0$ and $t<T$, the map $t\mapsto k_t
\in  \mathcal{K}_{\vartheta_t}\subset \mathcal{D}_{\vartheta_T}$,
cf. (\ref{28}), is continuous in $\mathcal{K}_{\vartheta_T}$ on
$[0,T)$, continuously differentiable on $(0,T)$ and satisfies
\[
\frac{d}{dt} k_t = L^\Delta_{\vartheta_T} k_t.
\]
\end{itemize}
\end{theorem}
The proof of this statement is performed in Section \ref{ESs} below.
\begin{corollary}
  \label{M1co}
There exists the map $[0,+\infty)\ni t
\mapsto \mu_t\in \mathcal{P}_{\rm exp}(\Gamma)$ such that $\mu_t$ is
the state of the population at time $t$. These states are
identified by their values $\mu_t(F^\theta) = \langle \! \langle
e(\theta;\cdot), k_t\rangle \!\rangle$ with $k_t$ as in Theorem \ref{1tm}, see (\ref{16}). Moreover,
for each $\theta \in \varTheta$, the map $(0,+\infty)\ni t \mapsto
\mu_t(F^\theta)\in\mathds{R}$ is differentiable and the following
holds
\[
\frac{d}{dt} \mu_t (F^\theta) = \langle \! \langle e(\theta;\cdot),
L^\Delta_{\vartheta_T} k_t\rangle \!\rangle,
\]
with any $T$ satisfying $T>t$.
\end{corollary}
Our next result establishes an upper bound for $\mu_t(N_\Lambda)$,
cf. (\ref{0}). For $r$ as in item (c) of Assumption \ref{ass}, we
set $\Delta_x=\{ y\in \mathds{R}^d: |x-y|\leq r/2\}$ and let
$\upsilon$ be the volume (Lebesgue's measure) of this ball.
\begin{theorem}
  \label{2tm}
Let $\mu_t$, $t>0$ be the state of the population as in Corollary \ref{M1co}.
In addition to items (a) and (b), assume that also item (c) of
Assumption \ref{ass} holds true. Then, for each compact
$\Lambda\subset \mathds{R}^d$ and $t>0$, the  following holds
\begin{equation}
  \label{32b}
  \mu_t(N_\Lambda) \leq \frac{m_\Lambda}{\phi_*} \log \left(
 C^{\phi_*}_{\Delta_0} (\mu_0) + (e^{\phi_*} -1) \bar{b} \upsilon t \right),
\end{equation}
where $m_\Lambda$ is the minimum number of the balls $\Delta_x$ that
cover $\Lambda$ and $C^{\phi_*}_{\Delta_0}(\mu_0)$ is given by the
right-hand side of (\ref{I4d}) with $\vartheta = \vartheta_0$, the
same  as in Theorem \ref{1tm}. If $\mu_0(\Gamma^{(0)})=1$, i.e., if
the system is initially empty, then one may take
$C^{\phi_*}_{\Delta_0}(\mu_0) =1$.
\end{theorem}

\subsubsection{The mesoscopic evolution}

As mentioned in Introduction, the mesoscopic theory of the evolution
of our model is based on the kinetic equation in (\ref{Z1}) the
translation invariant version of which is
\begin{equation}
  \label{8z}
\frac{d}{dt} \varrho_t (x) = b(x) \exp\bigg{(} - (\phi\ast
\varrho_t) (x) \bigg{)},
\end{equation}
where
\begin{equation}
  \label{9}
  (\phi\ast
\varrho_t) (x) := \int_{\mathds{R}^d} \phi(x-y) \varrho_t (y) dy.
\end{equation}
The definition of its classical solution is pretty similar to that
given in Definition \ref{1df}. Let $\varrho_0\in
L^\infty(\mathds{R}^d)$ be the initial condition in (\ref{8z}). Then
we set
\begin{equation}
  \label{c4}
 b^{+} = \langle \phi \rangle \esssup_{x\in \mathds{R}^d} b(x) e^{-(\phi*\varrho_0)
  (x)},  \quad  b^{-} = \langle \phi \rangle \essinf_{x\in \mathds{R}^d} b(x) e^{-(\phi*\varrho_0)
  (x)},
\end{equation}
and also
\begin{eqnarray}
  \label{c2}
\omega_{+} (t)& = & \omega_{-}(t) + (b^{+}-b^{-})t, \quad {\rm for}
\ \
b^{+} > b^{-}, \\[.2cm] \nonumber
\omega_{-} (t) & = & \log\left(\frac{b^{+}}{b^{+} - b^{-}} -
\frac{b^{-}}{b^{+} - b^{-}} e^{-(b^{+}-b^{-})t}
 \right), \quad b^{+} >b^{-}, \\[.2cm] \nonumber
\omega_{+} (t) & = & \omega_{-} (t) = \log \left(1 + b t \right),
\quad {\rm for} \ \  b^{+} = b^{-} = b.
\end{eqnarray}
\begin{theorem}
  \label{3tm}
For each $\varrho_0 \in L^\infty(\mathds{R}^d)$, the kinetic
equation in (\ref{8z}), (\ref{9}) has a unique global classical
solution $\varrho_t \in L^\infty(\mathds{R}^d)$ that for all $t>0$
and almost all $x\in \mathds{R}^d$ satisfies
\begin{equation}
  \label{c1}
\varrho_0 (x) + \frac{\omega_{-} (t)}{\langle \phi \rangle} \leq
\varrho_t(x) \leq \varrho_0 (x) + \frac{\omega_{+} (t)}{\langle \phi
\rangle}.
\end{equation}
\end{theorem}
\begin{theorem}
  \label{4tm}
Let $k_t$ and $\varrho_t$ be the solutions described by Theorems
\ref{1tm} and \ref{3tm}, respectively. Assume also that the initial
state $\mu_0$ in Theorem \ref{1tm} is Poisson approximable. That is,
there exist $\vartheta_0\in \mathds{R}$ and $q_{0,\varepsilon} \in
\mathcal{K}_{\vartheta_0}$, $\varepsilon \in (0,1]$  such that
$q_{0,1} = k_{\mu_0}$ and $\|q_{0,\varepsilon} -
k_{\pi_{\varrho_0}}\|_{\vartheta_0} \to 0$ as $\varepsilon \to
0^{+}$. Then there exist $\vartheta > \vartheta_0$ and $T>0$ such
that
\begin{equation}
\label{rho} \lim_{\varepsilon\to 0^{+}} \sup_{t\in [0,T]} \|q_{t,
\varepsilon}- k_{\pi_{\varrho_t}}\|_\vartheta = 0.
\end{equation}
\end{theorem}

\subsection{Comments}
\label{CommentsSS}

\subsubsection{The choice of the model}

Our choice of the translation invariant repulsion made in (\ref{17})
is motivated only by the convenience and simplicity of the
presentation. A more general version could contain $K(x,y)$ instead
of $\phi(x-y)$, cf. (\ref{Z1}). In this case, the theory developed
below would need a proper modification. As mentioned in
Introduction, the model can also be employed as a part of more
involved models. Let us mention some of them. As might be seen from
the example in (\ref{Z2}), (\ref{Z3}), the homogeneous solutions can
be unstable to small perturbations of the (homogeneous) rate $b$.
With this regard, a generalization of the model can be made by
allowing $b$ be dependent on some extra parameters, e.g., be random.
A possible choice might be
\begin{equation*}
  b(x) = b\exp\left(\sum_{y\in \omega} \Phi(x,y)\right),
\end{equation*}
where $\omega\subset \mathds{R}^d$ is the configuration of some
entities (attraction centers), distributed e.g., according to a
Poisson law, and $\Phi(x,y)$ is an attraction/repulsion potential.
Another possibility, close to just mentioned, is to consider a
two-type system of the Widom-Rowlinson type
\cite{asia1,FKKO,KKo,KKo1}. In this system, the particles of
different types would repel each other whereas those of the same
type not interact. Then instead of (\ref{Z1}) we would have
\begin{eqnarray*}
 \frac{d}{dt}\varrho_{0,t} (x) & = & b_0(x) \exp\left(- \int_{\mathds{R}^d}K(x,y) \varrho_{1,t}(y) dy
 \right), \\[.2cm]
 \frac{d}{dt}\varrho_{1,t} (x) & = & b_1(x) \exp\left(- \int_{\mathds{R}^d}K(x,y) \varrho_{0,t}(y) dy
 \right), \nonumber
\end{eqnarray*}
that strongly resemble the system in (\ref{Z2}), which manifests the
instability to appearing patterns discussed in Introduction. Another
possibility to modify the model is to include some motion of the
entities, e.g., random jumps as in the Kawasaki model \cite{asia}.

\subsubsection{Theorem \ref{1tm} and Corollary \ref{M1co}}

These statements establish the existence and uniqueness of the
global in time evolution of states through the evolution of the
corresponding correlation functions. In general, here we follow the
scheme developed and used in \cite{asia1,asia,KK2,Krzys}. Its main
aspects are outlined in subsection \ref{2.2ss}. Technically, the
proof of Theorem \ref{1tm} is the most involved part of this work,
divided into several steps. The hardest one is to prove that the
evolution $k_0\to k_t$ is such that $k_t$ be the correlation
function of a unique state. Usually, this link between correlation
functions and micro-state is not even discussed.

In the course of preparing the very formulation of Theorem \ref{1tm}
we have obtained the `true' hierarchy of evolution equations for
correlation functions of all orders mentioned in Introduction. It is
hidden in (\ref{15}) or (\ref{31}). To see it one has to recall that
$k(\{x_1, \dots, x_n\})= k^{(n)} (x_1 , \dots , x_n)$ and $k^{(n)}$
is the $n$-th order correlation function, cf. (\ref{7}). Then the
equation for $k_t^{(1)}$ (the first member of the hierarchy) reads,
see (\ref{8}) and (\ref{24}),
\begin{eqnarray*}
\frac{d}{dt} k_t^{(1)} (x) & = & b(x)\\[.2cm]  & + & b(x) \sum_{n=1}^\infty \frac{1}{n!} \int_{(\mathds{R}^d)^n
}\left(\prod_{i=1}^{n} [e^{-\phi(x-x_i)} -1]\right) k^{(n)}_t(x_1,
\dots , x_n) d x_1 \cdots d x_n. \nonumber
\end{eqnarray*}
In contrast to birth-and-death models \cite{Baker,FKKK,Moment}, this
equation contains correlation functions of all orders. Its `naive'
decoupling consists in setting $k^{(n)}_t(x_1, \dots , x_n)  \approx
k^{(1)}_t(x_1) \cdots k^{(1)}_t(x_n)$, which yields
\begin{equation*}
 \frac{d}{dt} k_t^{(1)} (x) = b(x) \exp\left( -
\int_{\mathds{R}^d}\left(1-e^{-\phi(x-y)}\right) k^{(1)}_t (y) dy
\right),
\end{equation*}
that can be considered as a version of (\ref{8z}), (\ref{9}). To
realize a more advanced decoupling here, one should write the next
equations in the hierarchy.

\subsubsection{Theorem \ref{2tm}}

This statement gives an example of the self-regulation of a complex
system due to interactions between its constituent, cf.
\cite{KK1,KK2}. As mentioned in Introduction and Section \ref{2S},
in the free version of our model (with $\phi=0$), the particle
density increases linearly in time, cf. (\ref{Poisson}). Then
Theorem \ref{2tm} shows that the repulsion satisfying item (c) of
Assumption \ref{ass} yields a significant attenuation of this
increase. This qualitatively corresponds to the case of $\alpha \geq
1$ in (\ref{Z2}) -- (\ref{Z4}) where the repulsion from the same
patch is more essential. Hence, one might speculate that under the
mentioned condition in item (c) the almost homogeneity of the
environment (encrypted in $b$) yields an almost homogeneous
distribution of entities. Or -- in other words -- the homogeneous
distribution is stable to small perturbations of the homogeneity of
the environment. However, this might be only a guess as we failed to
get any mathematical result in this direction. Moreover, the result
of Theorem \ref{2tm} cannot be transferred to the kinetic equation
(\ref{8z}), (\ref{9}). The reason for this is that in the proof we
crucially use the fact that the states $\mu_t$ satisfy (in the weak
sense, see (\ref{2})) the Kolmogorov equation (\ref{1}) with $L$
given in (\ref{17}). At the same time, there is no analogy of such
$L$ related to the kinetic equation, see also comments to Theorem
\ref{4tm}.

\subsubsection{Theorem \ref{3tm}}

The existence and uniqueness stated in this theorem are quite
expectable, and its most interesting result is the bounds
(\ref{c1}). Assume that the initial distribution of the entities is
tuned in such a way that $b^+=b^{-}=b$, see (\ref{c2}). Then the
solution is obtained explicitly in the form
\[
 \varrho_t (x) = \varrho_0(x) + \frac{1}{\langle \phi \rangle} \log(1 + bt),
\]
that is, its $x$-dependence repeats that of $\varrho_0(x)$. By
(\ref{c4}) the mentioned tuning consists in satisfying
\[
 \forall x\in \mathds{R}^d \ \qquad b(x) \exp\left(- (\phi*\varrho_0)(x) \right) ={\rm const},
\]
and thus should take into account the interplay between $b$ and
$\phi$. If $b^{+}>b^{-}$, then the function $\omega_-(t)$ is bounded
in time whereas $\omega_{+}(t)$ increases linearly. This resembles
the situation in (\ref{Z2}) with $\alpha<1$. Note that $\phi$ is
Theorem \ref{3tm} in not supposed to satisfy item (c) of Assumption
\ref{ass}. Unfortunately, we failed to get more precise bounds at
the expense of further restriction imposed on $\phi$.

\subsubsection{Theorem \ref{4tm}}

This statement establishes a relationship between the mesoscopic
evolution $\varrho_0\to \varrho_t$ obtained from the kinetic
equation and the Markov evolution of micro-states $\mu_0\to \mu_t$.
Its main message is that the former evolution approximates the
latter one in the sense of Definition \ref{Poidf}. And its main
drawback is that this relationship is restricted in time to $[0,T]$
with $T$ found explicitly, see the the proof of the theorem below.
Here one has to keep in mind that the evolution of the Poisson
states $\pi_{\varrho_0} \to \pi_{\varrho_t}$ corresponding to the
mentioned evolution of the densities need not be Markov, that is, it
cannot be obtained from equations like that in (\ref{1}). The second
thing to be aware of is that $q_{t,\varepsilon}$ in (\ref{rho}) with
$\varepsilon\in (0,1)$  need not be correlation functions. Because
of this, their continuation in time were impossible that yields the
drawback mentioned above.

\section{The Evolution of States}
\label{ESs}

In this section, we prove Theorem \ref{1tm} and Corollary
\ref{M1co}. We begin by proving that (\ref{31}) has a unique
solution on a bounded time interval.
\subsection{Finite time horizon}
For $\vartheta \in \mathds{R}$ and $\vartheta'>\vartheta$, we set
\begin{equation}
  \label{33}
T(\vartheta',\vartheta)= \frac{\vartheta'-\vartheta}{\bar{b}}
\exp\left(\vartheta - \langle \phi \rangle e^{\vartheta'} \right) =
\frac{\vartheta'-\vartheta}{\beta(\vartheta')}e^\vartheta,
\end{equation}
see (\ref{29}).
\begin{lemma}
  \label{F1lm}
For each $\vartheta_0\in \mathds{R}$ and $\vartheta>\vartheta_0$,
the problem in (\ref{31}) with $k_0\in \mathcal{K}_{\vartheta_0}$
has a unique solution on the time interval
$[0,T(\vartheta,\vartheta_0))$.
\end{lemma}
\begin{proof}
For an arbitrary $\vartheta_1\in \mathds{R}$ and $\vartheta_2
>\vartheta_1$, by means of (\ref{24}) one can define a bounded operator
$(L^\Delta)^2_{\vartheta_2\vartheta_1} : \mathcal{K}_{\vartheta_1}
\to \mathcal{K}_{\vartheta_2}$ the operator norm of which can be
estimated similarly as in (\ref{25}) and (\ref{26}). Clearly, for
each $\vartheta'\in (\vartheta_1,\vartheta_2)$, it satisfies
\begin{equation}
  \label{34}
  \|(L^\Delta)^2_{\vartheta_2\vartheta_1} \| \leq \frac{\beta
  (\vartheta') \beta (\vartheta_1)e^{-\vartheta'-\vartheta_1}}{e^2 (\vartheta_2 -
  \vartheta')(\vartheta'-\vartheta_1)} \leq \frac{[\beta
  (\vartheta_2)]^2e^{-2\vartheta_1}}{e^2 (\vartheta_2 -
  \vartheta')(\vartheta'-\vartheta_1)}.
\end{equation}
Note that the definition of $(L^\Delta)^2_{\vartheta_2\vartheta_1}$
is independent of the choice of $\vartheta'$, which we use to
estimate the norm of this operator. In a similar way, one defines
$(L^\Delta)^n_{\vartheta_2\vartheta_1}$ for all $n\in  \mathds{N}$.
To estimate its norm, we introduce
\begin{equation*}
 \vartheta^l = \vartheta_1 + l\epsilon, \qquad \epsilon = (\vartheta_2 -
 \vartheta_1)/n.
\end{equation*}
Then like in (\ref{34}) we obtain
\begin{equation}
  \label{36}
  \|(L^\Delta)^n_{\vartheta_2\vartheta_1}\| \leq
  \left(\frac{n}{e}\right)^n
  \frac{1}{[T(\vartheta_2,\vartheta_1)]^n}.
\end{equation}
Let us consider $(L^\Delta)^0_{\vartheta_2\vartheta_1}$ as the
corresponding (continuous) embedding operator, cf. (\ref{23}). Then,
for each $\vartheta'\in (\vartheta_1 , \vartheta_2)$, $n\in
\mathds{N}_0$ and $m\leq n$, we have that
\begin{equation}
  \label{36a}
(L^\Delta)^n_{\vartheta_2\vartheta_1} =
(L^\Delta)^m_{\vartheta_2\vartheta'}
(L^\Delta)^{n-m}_{\vartheta'\vartheta_1} =
(L^\Delta)^0_{\vartheta_2\vartheta'}
(L^\Delta)^n_{\vartheta'\vartheta_1} .
\end{equation}
In view of (\ref{28}) and (\ref{30}), these equalities imply
\begin{eqnarray}
  \label{36b}
  & & \forall n\in \mathds{N}_0 \qquad
  (L^\Delta)^n_{\vartheta_2\vartheta_1}: \mathcal{K}_{\vartheta_1}
  \to \mathcal{D}_{\vartheta_2}, \\[.2cm] \nonumber  & & \forall n\in \mathds{N} \qquad
  (L^\Delta)^n_{\vartheta_2\vartheta_1} = L^\Delta_{\vartheta_2}
  (L^\Delta)^{n-1}_{\vartheta_2\vartheta_1}.
\end{eqnarray}
Consider
\begin{equation}
  \label{37}
Q_{\vartheta_2 \vartheta_1} (t) = \sum_{n\geq 0}\frac{t^n}{n!}
(L^\Delta)^n_{\vartheta_2\vartheta_1}.
\end{equation}
By (\ref{36}) we conclude that the series in (\ref{37}) converges in
the operator norm topology -- uniformly on compact subsets of
$[0,T(\vartheta_2, \vartheta_1))$ -- and thus defines a continuous
function
\begin{equation}
  \label{37a}
[0,T(\vartheta_2, \vartheta_1)) \ni t \mapsto Q_{\vartheta_2
\vartheta_1} (t) \in \mathcal{C}_{\vartheta_2\vartheta_1},
\end{equation}
where $\mathcal{C}_{\vartheta_2\vartheta_1}$ denotes that Banach
space of all bounded linear operators acting from
$\mathcal{K}_{\vartheta_1}$ to $\mathcal{K}_{\vartheta_2}$. Its
operator norm can be estimated with the help of (\ref{36}), that
yields
\begin{equation}
  \label{37b}
 \|Q_{\vartheta_2
\vartheta_1} (t)\| \leq \frac{T(\vartheta_2,
\vartheta_1)}{T(\vartheta_2, \vartheta_1)-t}.
\end{equation}
In a similar way, by (\ref{36a}) and the second equality in
(\ref{36b}) one obtains that
\begin{equation}
  \label{38}
\forall t \in (0, T(\vartheta_2, \vartheta_1)) \qquad   \frac{d}{dt}
Q_{\vartheta_2 \vartheta_1} (t) = L^\Delta_{\vartheta_2\vartheta'}
Q_{\vartheta' \vartheta_1}
  (t) =
  L^\Delta_{\vartheta_2} Q_{\vartheta_2 \vartheta_1} (t),
\end{equation}
where the choice of $\vartheta'\in (\vartheta_1, \vartheta_2)$
depends on the value of $t$ and should be made in such a way that
$t<T(\vartheta',\vartheta_1)$, which is possible since
$T(\vartheta',\vartheta_1)$ continuously depends on $\vartheta'$. By
(\ref{38}) we conclude that the map in (\ref{37a}) is continuously
differentiable on $(0, T(\vartheta_2, \vartheta_1))$. We combine all
these facts and obtain that
\begin{equation}
  \label{39}
 k_t = Q_{\vartheta \vartheta_0} (t) k_0
\end{equation}
is the solution in question, see Definition \ref{1df}. Assume now
that $\hat{k}_t$ is another solution of (\ref{31}) with the same
initial condition. Then $u_t = \hat{k}_t - k_t\in \mathcal{K}_\vartheta$ also solves this
problem but with the zero initial condition. Let $\tau\geq 0$ be such that
$u_t = 0$ for $t\leq \tau$. Take $t\in(\tau, T(\vartheta,\vartheta_0))$ and write
\begin{equation}
 \label{40}
 I_{\vartheta'\vartheta}u_t = \int_\tau^t L^\Delta_{\vartheta'\vartheta} u_s ds = L^\Delta_{\vartheta'\vartheta}\int_\tau^t u_s ds,
\end{equation}
which obviously holds with any $\vartheta'>\vartheta$. Here
$I_{\vartheta'\vartheta}:\mathcal{K}_{\vartheta} \to
\mathcal{K}_{\vartheta'}$ is the embedding operator, cf. (\ref{23}).
Now we repeat in (\ref{40}) the same equality for $u_s$, and then
repeat this again due times to get
\begin{eqnarray}
 \label{41}
 I_{\vartheta'\vartheta}u_t & = & ( L^\Delta)^2_{\vartheta'\vartheta}\int_\tau^t \left( \int_\tau^{t_1} u_s ds \right) d t_1\\[.2cm] \nonumber & = &
 ( L^\Delta)^n_{\vartheta'\vartheta}\int_\tau^t\int_\tau^{t_1} \cdots \int_\tau^{t_{n-1} }\left( \int_\tau^{t_1} u_s ds \right) d t_1 \cdots d t_{n-1},
\end{eqnarray}
with an arbitrary $n\in \mathds{N}$.
Then by means of (\ref{36}) we obtain from (\ref{41}) that
\[
\|
I_{\vartheta'\vartheta}u_t\|_{\vartheta'} \leq \frac{1}{n!} \left(\frac{n}{e} \right)^n \left(\frac{t- \tau}{T(\vartheta', \vartheta)}\right)^n
\sup_{s\in [\tau, t]} \|u_s\|_\vartheta.
\]
Since $n$ here is an arbitrary integer, we have that
$I_{\vartheta'\vartheta}u_t = 0$ (hence, $u_t=0$) for all those
values of $t\in(\tau, T(\vartheta,\vartheta_0))$, for which $t-
\tau<T(\vartheta', \vartheta)$ with a given $\vartheta'>\vartheta$.
In such a way, one proves that $u_t=0$ for all $t\in[0,
T(\vartheta,\vartheta_0))$, which completes the whole proof.
\end{proof}
\begin{remark}
 \label{Frk}
Since $Q_{\vartheta_2\vartheta_1}(t)$ is defined by the series in
(\ref{37}), by (\ref{30}) one can prove that, for each $\vartheta_1'
\leq \vartheta_1$ and $\vartheta_2' \in(\vartheta_1', \vartheta_2)$,
the follows holds
\[
 \forall k\in \mathcal{K}_{\vartheta_1'} \qquad Q_{\vartheta_2 \vartheta_1} (t) k= Q_{\vartheta_2' \vartheta_1'} (t) k,
\]
whenever $t < \min\{T(\vartheta_2, \vartheta_1); T(\vartheta_2', \vartheta_1')\}$.
\end{remark}

\subsection{The identification}

The main result of this section is given in the following statement.
Recall that $\mathcal{K}^\star_\vartheta$ is defined in (\ref{32}).
\begin{lemma}
  \label{F2lm}
For each $t<T(\vartheta_2,\vartheta_1)/2$, the operator defined in
(\ref{37}) has the property $Q_{\vartheta_2\vartheta_1}(t) :
\mathcal{K}^\star_{\vartheta_1} \to
\mathcal{K}^\star_{\vartheta_2}$.
\end{lemma}
By this lemma it will follow that the solution $k_t$ obtained in
Lemma \ref{F1lm} for $t<T(\vartheta_2,\vartheta_1)/2$ satisfies all
the conditions of Proposition \ref{Tobipn} -- and hence is the
correlation function of a unique $\mu\in \mathcal{P}_{\rm
exp}(\Gamma)$ -- whenever $k_0$ has the same property. Its proof is
performed in a number of steps. First we introduce an auxiliary
model by replacing $b(x)$ by and integrable  $b^\sigma(x)$,
$\sigma>0$. For this new model, by repeating the steps made to prove
Lemma \ref{F1lm} we construct $Q^\sigma_{\vartheta_2\vartheta_1}$,
for which we prove that $Q^\sigma_{\vartheta_2\vartheta_1}(t) :
\mathcal{K}^\star_{\vartheta_1} \to
\mathcal{K}^\star_{\vartheta_2}$. Then we show that
\begin{equation}
  \label{42}
  \langle \!\langle G,Q^\sigma_{\vartheta_2\vartheta_1} (t) k\rangle
  \!\rangle \to \langle \!\langle G,Q_{\vartheta_2\vartheta_1} (t) k\rangle
  \!\rangle, \qquad {\rm as}  \ \ \sigma \to 0^+.
\end{equation}

\subsubsection{An auxiliary model}

For $\sigma >0$, we set
\begin{equation}
  \label{43}
  \psi_\sigma (x) = \exp\left( - \sigma |x|^2 \right), \qquad
  b^\sigma (x) = b(x) \psi_\sigma (x).
\end{equation}
Let $L^{\Delta,\sigma}$ stand for the corresponding operator defined
in (\ref{24}) with $b$ replaced by $b^\sigma$. Clearly, this
$L^{\Delta,\sigma}$ satisfies the estimates in (\ref{29}) and
(\ref{36}) with the same right-hand sides, that allows one to define
$Q_{\vartheta_2\vartheta_1}^\sigma (t)$ by the series as in
(\ref{37}) with $L^{\Delta}$ replaced by $L^{\Delta,\sigma}$.
Likewise, one defines also $L^{\Delta,\sigma}_\vartheta$ and
$L^{\Delta,\sigma}_{\vartheta'\vartheta}$. Note that the operator
norm of $Q_{\vartheta_2\vartheta_1}^\sigma (t)$ satisfies
(\ref{37b}) with the same right-hand side.

The Banach space predual to $\mathcal{K}_\vartheta$ is
\begin{equation}
  \label{44}
  \mathcal{G}_\vartheta =\{ G:\Gamma_0\to \mathds{R}:
  |G|_\vartheta<\infty\}, \qquad |G|_\vartheta :=\int_{\Gamma_0}
  |G(\eta)|e^{\vartheta|\eta|} \lambda ( d\eta).
\end{equation}
Like in (\ref{23}) we then have
\begin{equation}
  \label{44a}
\mathcal{G}_{\vartheta'} \hookrightarrow \mathcal{G}_\vartheta,
\qquad {\rm for} \ \ \vartheta'>\vartheta.
\end{equation}
Let $\widehat{L}^\sigma$ be defined by the relation
\begin{equation}
  \label{46}
    \langle \!\langle \widehat{L}^\sigma G, k\rangle
  \!\rangle = \langle \!\langle G, L^{\Delta,\sigma }k\rangle
  \!\rangle .
\end{equation}
Then we set, cf. (\ref{27}),
\begin{equation*}
 \widehat{D}_\vartheta = \{ G\in \mathcal{G}_\vartheta: |\cdot |G
 \in  \mathcal{G}_\vartheta\}.
\end{equation*}
This allows one to define
$\widehat{L}^\sigma_\vartheta=(\widehat{L}^\sigma,
\widehat{D}_\vartheta)$, and also bounded operators
$(\widehat{L}^\sigma)^n_{\vartheta\vartheta'}:
\mathcal{G}_{\vartheta'} \to \mathcal{G}_\vartheta$, $n\in
\mathds{N}$, the operator norms of which satisfy the estimates in
(\ref{36}),  with the same right-hand side. Thereafter, we define,
cf. (\ref{37}),
\begin{equation}
  \label{48}
  H^\sigma_{\vartheta_1 \vartheta_2} (t) = \sum_{n\geq
  0}\frac{t^n}{n!}(\widehat{L}^\sigma)^n_{\vartheta_1\vartheta_2},
  \qquad \vartheta_2>\vartheta_1.
\end{equation}
Like in the proof of Lemma \ref{F1lm} one shows that the series in
(\ref{48}) defines a continuous map
\begin{equation}
  \label{49}
[0,T(\vartheta_2, \vartheta_1)) \ni t \mapsto H^\sigma_{\vartheta_1
\vartheta_2} (t) \in \widehat{\mathcal{C}}_{\vartheta_1\vartheta_2},
\end{equation}
where $\widehat{\mathcal{C}}_{\vartheta_1\vartheta_2}$ is the Banach
space of all bounded linear operators acting from
$\mathcal{G}_{\vartheta_2}$ to $\mathcal{G}_{\vartheta_1}$. The
operator norm of $H_{\vartheta_1\vartheta_2}^\sigma (t)$ satisfies
(\ref{37b}) with the same right-hand side. By the very construction,
cf. (\ref{46}), we have that
\begin{equation}
  \label{50}
  \forall t \in [0,T(\vartheta_2, \vartheta_1)) \qquad     \langle \!\langle H^\sigma_{\vartheta_1 \vartheta_2} (t) G, k\rangle
  \!\rangle = \langle \!\langle G, Q^\sigma_{\vartheta_2\vartheta_1}
  (t)k\rangle,
  \!\rangle .
\end{equation}
holding for all $k\in \mathcal{K}_{\vartheta_1}$ and $G\in
\mathcal{G}_{\vartheta_2}$.

In the sequel, we will use the following property of
$\widehat{L}^\sigma$. To describe it, we derive from (\ref{46}) and
(\ref{24}), (\ref{25}) the explicit form of this operator. To this
end, we employ (\ref{Minlos}) and obtain
\begin{equation}
  \label{50a}
( \widehat{L}^\sigma G)(\eta) = \int_{\mathds{R}^d} b^\sigma (x)
\sum_{\xi \subset \eta} e(t_x;\xi) e(\tau_x;\eta\setminus
\xi)G(\eta\setminus \xi \cup x) d x.
\end{equation}
Note that, for a broad class of functions $G:\Gamma_0 \to
\mathds{R}$, e.g., for $G\in B_{\rm bs}(\Gamma_0)$, the computations
in (\ref{50a}) can be performed explicitly as the integral is
convergent and the sum is finite. In Appendix, we prove that
\begin{equation}
  \label{50b}
  L^\sigma K G = K \widehat{L}^\sigma G,
\end{equation}
where $K$ is defined in (\ref{7z}). Then, for each $G\in B_{\rm
bs}(\Gamma_0)$ and $n\in \mathds{N}$, we obtain from (\ref{50b}) the
following
\begin{equation*}
(L^\sigma)^n K G = K (\widehat{L}^\sigma)^n G
\end{equation*}

\subsubsection{Approximations}

Now we aim at proving that
\begin{equation}
  \label{51}
\forall \sigma >0 \quad \forall t\in [0,T(\vartheta_2, \vartheta_1))
\qquad   \langle \!\langle G,Q^\sigma_{\vartheta_2\vartheta_1} (t)
k\rangle
  \!\rangle \geq 0,
\end{equation}
whenever $k\in \mathcal{K}_{\vartheta_1}^\star$ and $G\in
B^\star_{\rm bs}(\Gamma_0)$. Note that $B_{\rm bs}(\Gamma_0)\subset
\mathcal{G}_\vartheta$ for any real $\vartheta$, see Definition
\ref{Gdef}. To prove (\ref{51}) we approximate $k$ by
$k^{\Lambda,N}$, $\Lambda$ and $N$ being a compact subset of
$\mathds{R}^d$ and and integer, respectively. The meaning of this is
that $Q^\sigma_{\vartheta_2\vartheta_1} (t) k^{\Lambda,N}$ has the
desired positivity by construction. Then the proof of (\ref{51})
will be done by showing the corresponding convergence as $\Lambda
\to \mathds{R}^d$ and $N\to +\infty$.

For a compact $\Lambda$, let $\Gamma_\Lambda$ denote the set of all
$\gamma\in \Gamma$ such that $\gamma\subset \Lambda$. Clearly,
$\Gamma_\Lambda \subset \Gamma_0$ and $\Gamma_\Lambda \in
\mathcal{B}(\Gamma)$. Set $\mathcal{B}(\Gamma_\Lambda)= \{ A\in
\mathcal{B}(\Gamma): A\subset \Gamma_\Lambda\}$ and $p_\Lambda
(\gamma) = \gamma\cap\Lambda$, $\gamma\in \Gamma$. Then
\begin{equation}
  \label{52}
  \mu^\Lambda (A) = \mu(p_\Lambda^{-1} (A) ), \qquad A \in
  \mathcal{B }(\Gamma_\Lambda)
\end{equation}
defines a probability measure on the measurable space
$(\Gamma_\Lambda, \mathcal{B }(\Gamma_\Lambda))$ -- the projection
of $\mu$ on $\Gamma_\Lambda$. Here $p_\Lambda^{-1} (A) =\{ \gamma
\in \Gamma: p_\Lambda (\gamma)\in A\}$. It is possible to show, see
\cite{Tobi}, that for each  $\mu\in \mathcal{P}_{\rm exp}(\Gamma)$
and any compact $\Lambda$, $\mu^\Lambda$ is absolutely continuous
with respect to $\lambda$. For $\mu\in \mathcal{P}_{\rm
exp}(\Gamma)$, let $R^\Lambda_\mu:\Gamma_\Lambda \to \mathds{R}$ be
the corresponding Radon-Nikodym derivative. The the correlation
function $k_\mu$ and $R^\Lambda_\mu$ are related to each other by
\begin{equation}
  \label{52a}
  k_\mu (\eta) = \int_{\Gamma_\Lambda} R^\Lambda_\mu (\eta \cup
  \xi)\lambda ( d \xi), \quad \eta \in \Gamma_\Lambda.
\end{equation}
For a given $N\in \mathds{N}$ and $\eta\in \Gamma_0$, we then set
\begin{equation}
  \label{53}
R^{\Lambda,N}_0 (\eta) =\left\{\begin{array}{ll} R^\Lambda_\mu
(\eta) \qquad &{\rm if}
\quad \eta \in \Gamma_\Lambda \ {\rm and } \ |\eta|\leq N;\\[.3cm] 0
\qquad &{\rm otherwise}.
\end{array} \right.
\end{equation}
Clearly, we have that $R^{\Lambda,N}_0 (\eta) \geq 0$ and
$R^{\Lambda,N}_0 \in \mathcal{G}_\vartheta$ for any $\vartheta\in
\mathds{R}$. Set, cf. (\ref{52a}),
\begin{equation}
 \label{54}
 q_0^{\Lambda,N} (\eta) = \int_{\Gamma_0} R_0^{\Lambda,N} (\eta\cup \xi) \lambda ( d \xi), \quad \eta \in \Gamma_0.
\end{equation}
Then, for $G\in B^\star_{\rm bs}(\Gamma_0)$, by (\ref{7z}) and (\ref{Minlos}) we have that
\begin{equation}
 \label{55}
\langle \! \langle G, q^{\Lambda,N}_0 \rangle \! \rangle = \int_{\Gamma_0} (KG)(\eta) R^{\Lambda,N}_0 (\eta) \lambda (d\eta) \geq 0.
\end{equation}
At the same time, by (\ref{52a}), (\ref{53}) and (\ref{54}) we have
\begin{equation}
  \label{55a}
0\leq  q_0^{\Lambda,N} (\eta) \leq k_\mu(\eta),
\end{equation}
holding for $\lambda$-almost all $\eta \in \Gamma_0$.

By (\ref{53}) $R^{\Lambda,N}_0$ (up to normalization) is the density
of a state in which the number of entities does not exceed $N$, and
hence is finite. We allow it to evolve $R^{\Lambda,N}_0\to
R^{\Lambda,N}_t$, where $R^{\Lambda,N}_t$ is obtained from the
Fokker-Planck equation
\begin{equation*}
 \frac{d}{dt} R^{\Lambda,N}_t = L^\dagger R^{\Lambda,N}_t, \qquad R^{\Lambda,N}_t|_{t=0} = R^{\Lambda,N}_0,
\end{equation*}
where $L^\dagger$ is defined by
\begin{equation}
 \label{56}
 \int_{\Gamma_0} (L^\sigma F)(\eta) R(\eta) \lambda(d \eta) =  \int_{\Gamma_0} F(\eta)(L^\dag R)(\eta) \lambda(d \eta).
\end{equation}
In Appendix, we show that
\begin{gather}
 \label{57}
 L^\dag = A + B,\\[.2cm] \nonumber
 (A R) (\eta) = - \Psi_\sigma (\eta) R(\eta), \quad \Psi_\sigma (\eta):=
 \int_{\mathds{R}^d} b^\sigma (x) e(\tau_x;\eta) d x,\\[.2cm]\nonumber
 (B R)(\eta) =   \sum_{x\in \eta} b^\sigma (x) e(\tau_x;\eta\setminus x)R(\eta\setminus x).
\end{gather}
Note that the reason to substitute $b$ by $b^\sigma$ is to make convergent the integral in the second line of (\ref{57}).
Note also that according to our assumptions and (\ref{43}) we have that
\begin{equation}
 \label{58}
 \Psi_\sigma (\eta) \leq \bar{b} \left( \frac{\pi}{\sigma}\right)^{d/2},
\end{equation}
which means that $A$ is a bounded multiplication operator. Now we define $L^\dag$ in $\mathcal{G}_0$, introduced in
(\ref{44}) with $\vartheta=0$. For $R\in \mathcal{G}_0$, by (\ref{57}) we get
\begin{gather}
 \label{59}
 |BR|_0 \leq \int_{\Gamma_0} \left( \sum_{x\in \eta} b^\sigma (x) e(\tau_x;\eta\setminus x) |R(\eta\setminus x)|\right)|
 \lambda ( d\eta)\\[.2cm] \nonumber
 = \int_{\Gamma_0} \Psi_\sigma (\eta)| R(\eta)| \lambda (d \eta) \leq \bar{b} \left( \frac{\pi}{\sigma}\right)^{d/2} |R|_0.
\end{gather}
This means that $L^\dag$ is bounded and $\|L^\dag\| \leq 2
\bar{b}(\pi/\sigma)^{d/2}$. Let $\mathcal{G}_0^{+}$ be the cone of
positive elements of $\mathcal{G}_0$. For $R\in \mathcal{G}_0^{+}$,
by making the same calculations as in (\ref{59}) we obtain
\begin{equation*}
  \int_{\Gamma_0} (L^\dag R)(\eta) \lambda ( d \eta) = 0
\end{equation*}
Therefore, $L^\dag$ generates a stochastic semigroup $S^\dag =
\{S^\dag (t)\}_{t \geq 0}$ on $\mathcal{G}_0$. This semigroup has a useful property which we
describe now. Recall that $\mathcal{G}_\vartheta$ is defined in (\ref{44}) and that $\mathcal{G}_\vartheta\hookrightarrow \mathcal{G}_0$ whenever $\vartheta>0$,
see (\ref{44a}).
\begin{lemma}
 \label{60lm}
For  each $\vartheta>0$ and $t>0$, we have that $S^\dag (t) : \mathcal{G}_\vartheta \to \mathcal{G}_\vartheta$.
\end{lemma}
\begin{proof}
We employ the Thieme-Voigt perturbation theory in the form adapted
to the present context formulated in \cite[Proposition 3.2, page
421]{asia}. According to item (iv) of this statement, we have to
show that there exist positive $c$ and $\varepsilon$ such that, for
all positive $R\in \mathcal{G}_\vartheta$, the following holds
\begin{gather}
\label{60a}
 \int_{\Gamma_0} e^{\vartheta |\eta|} (L^\dag R) (\eta) \lambda ( d \eta) \leq
c \int_{\Gamma_0} e^{\vartheta |\eta|} R (\eta) \lambda ( d \eta)- \varepsilon  \int_{\Gamma_0} \Psi_\sigma (\eta) R (\eta) \lambda ( d \eta).
\end{gather}
Set $F_\vartheta(\eta) = e^{\vartheta|\eta|}$. Then by (\ref{56}) we have
\begin{eqnarray*}
 {\rm LHS(\ref{60a})}& = & \int_{\Gamma_0} (L^\sigma F_\vartheta)(\eta) R (\eta) \lambda ( d \eta)\\[.2cm] \nonumber
& = & (e^\vartheta -1) \int_{\Gamma_0} \left(\int_{\mathds{R}^d} b^\sigma (x) e(\tau_x;\eta) dx \right)e^{\vartheta|\eta|} R (\eta)
\lambda ( d \eta)\\[.2cm] \nonumber &\leq & \bar{c}(\sigma, \vartheta) \int_{\Gamma_0} e^{\vartheta |\eta|} R (\eta) \lambda ( d \eta).
\end{eqnarray*}
Here $\bar{c}(\sigma, \vartheta):= (e^\vartheta -1) \bar{b} (\pi/\sigma)^d$, cf. (\ref{58}). Then (\ref{60a}) holds with
$\varepsilon\leq  (\sigma/\pi)^d \bar{b}^{-1}$ and $c \geq 1 + \bar{c}(\sigma, \vartheta)$.
\end{proof}
By (\ref{60a}) and the Gr\"onwall inequality we obtain that
\begin{equation}
 \label{60c}
 \int_{\Gamma_0} e^{\vartheta |\eta|} \left(S^\dag (t)R^{\Lambda,N}_0\right) (\eta) \lambda ( d \eta) \leq \exp\bigg{(} \vartheta N +
 \bar{c}(\sigma, \vartheta)t \bigg{)},
\end{equation}
holding for all $t>0$, see (\ref{53}).
\begin{lemma}
  \label{F3lm}
Let $\mu_0\in \mathcal{P}_{\rm exp}(\Gamma)$ and $\vartheta_0\in
\mathds{R}$ be such that $k_{\mu_0}\in \mathcal{K}_{\vartheta_0}$.
Let also $\vartheta > \vartheta_0$ be fixed. Finally, let
$R^{\Lambda,N}_0$ be calculated according to (\ref{54}), (\ref{53})
with this $\mu_0$. Then, for each $G\in B_{\rm bs}(\Gamma_0)$ and
all $t< T(\vartheta, \vartheta_0)$, $\sigma
>0$, compact $\Lambda$ and $N\in \mathds{N}$, the following holds
\begin{equation}
 \label{61}
 \langle \!  \langle G, Q^\sigma_{\vartheta \vartheta_0}(t) q_0^{\Lambda,N} \rangle \! \rangle = \langle \! \langle  KG , S^\dag (t)
 R^{\Lambda,N}_0 \rangle
 \!
 \rangle,
\end{equation}
where $KG$ is as in (\ref{7z}).
\end{lemma}
\begin{proof}
For a given $G\in B_{\rm bs}(\Gamma_0)$ and $t< T(\vartheta,
\vartheta_0)$, denote
\begin{equation}
  \label{62}
  f_G (t) = \langle \!  \langle G, Q^\sigma_{\vartheta \vartheta_0}(t) q_0^{\Lambda,N} \rangle \!
  \rangle, \quad g_G(t) = \langle \! \langle  KG , S^\dag (t)
 R^{\Lambda,N}_0 \rangle
 \!
 \rangle,
\end{equation}
For this $t$, one finds $\vartheta_1\in (\vartheta_0,\vartheta)$
such that $t< T(\vartheta_1, \vartheta_0)$. Then, for each $n\in
\mathds{N}$, by (\ref{38}) we obtain
\begin{equation}
  \label{63}
f_G^{(n)}(t) := \frac{d^n f_G(t)}{dt^n} = \langle \!  \langle G,
(L^{\Delta,\sigma})^n_{\vartheta, \vartheta_1} Q^\sigma_{\vartheta_1
\vartheta_0}(t) q_0^{\Lambda,N} \rangle \!
  \rangle,
\end{equation}
We take now any $\tau< T(\vartheta, \vartheta_0)$ and then pick
$\vartheta_1\in (\vartheta_0,\vartheta)$ such that $\tau<
T(\vartheta_1, \vartheta_0)$. Then we apply in (\ref{63}) the
estimates in (\ref{36}), (\ref{37b}) and (\ref{55a}) to obtain
\begin{gather}
  \label{64}
 \sup_{t\in [0,\tau]}  \left| \frac{d^n f_G(t)}{dt^n} \right|  \leq
   C_fn^n M_f^n, \\[.2cm] \nonumber
 C_f  :=  |G|_\vartheta \sup_{t\in
 [0,\tau]}\|Q^{\sigma}_{\vartheta_1\vartheta_0}(t) q^{\Lambda,N}_0
 \|_{\vartheta_1}\leq \frac{T(\vartheta_1, \vartheta_0)}{T(\vartheta_1, \vartheta_0)-\tau}|G|_\vartheta \|k_{\mu_0}\|_{\vartheta_0},  \\[.2cm] \nonumber
 M_f = 1/ e T(\vartheta_1,
 \vartheta_0).
\end{gather}
Likewise, for the derivatives of  $g_G$, we get
\begin{equation}
 \label{64a}
 g_G^{(n)}(t) := \frac{d^n g_G (t)}{d t^n} = \langle \! \langle KG , (L^\dag)^n S^\dag (t) R^{\Lambda,N}_0 \rangle \! \rangle.
\end{equation}
Recall that $G$ here belongs to $B_{\rm bs}(\Gamma_0)$, see Definition \ref{Gdef}.
Let $C_G>0$ be such that $|G(\eta)|\leq C_G$ for all $\eta$. Then
\[
 |(K G)(\eta) | \leq C_G \exp\left(\omega |\eta| \right), \qquad \omega:= \log 2.
\]
We use this estimate in (\ref{64a}) together with (\ref{60c}) and the fact that $L^\dag$ -- the generator of $S^\dag$ -- is bounded,
to obtain that, for the same $\tau$ as in (\ref{64}), the following holds
\begin{gather}
  \label{65}
 \sup_{t\in [0,\tau]}  \left| \frac{d^n g_G(t)}{dt^n} \right|  \leq
   C_g M_g^n, \\[.2cm] \nonumber
 C_g  :=  C_G \exp \bigg{(}\omega N + \bar{c}(\sigma, \omega) \tau \bigg{)},  \\[.2cm] \nonumber
 M_f = 2 \bar{b}\left(\frac{\pi}{\sigma} \right)^{d/2}.
\end{gather}
By (\ref{65}) we conclude that $g_G$ is a real analytic function on
each $[0,\tau]$. At the same time, by the Denjoy-Carleman theorem
\cite{DC} one gets that $f_G$ is quasi-analytic on each $[0,\tau]$,
$\tau< T(\vartheta,\vartheta_0)$. These two facts imply that $h_G =
f_G - g_G$ is also quasi-analytic on each $[0,\tau]$ with $\tau$
chosen as mentioned above. Now by (\ref{46}) and (\ref{63}) we get
\begin{equation}
 \label{66}
 f_G^{(n)}(0) = \langle \! \langle ( \widehat{L}^\sigma )^n G, q^{\Lambda,N}_0\rangle \! \rangle .
\end{equation}
Likewise, by (\ref{56}) and (\ref{64a}) we have
\begin{equation*}
 g_G^{(n)}(0) = \langle \! \langle ( L^\sigma )^n KG, R^{\Lambda,N}_0\rangle \! \rangle .
\end{equation*}
Now we apply here (\ref{50b}), then (\ref{55}) and arrive at
\begin{gather*}
 g_G^{(n)}(0) = \langle \! \langle  K ( \widehat{L}^\sigma )^n G, R^{\Lambda,N}_0\rangle \! \rangle
 =  \langle \! \langle ( \widehat{L}^\sigma )^n G, q^{\Lambda,N}_0\rangle \! \rangle,
\end{gather*}
which by (\ref{66}) yields that $h_G^{(n)}(0)=0$ holding for all $n\in \mathds{N}$. Since
$h_G$ is quasi-analytic, this implies that $h_G(t)=0$ for all $t<T(\vartheta, \vartheta_0)$, that by (\ref{62}) completes the proof.
\end{proof}
\begin{corollary}
 \label{Fco}
 Let $q_0^{\Lambda,N}$ and $\vartheta$, $\vartheta_0$, $t$ be as in Lemma \ref{F3lm}.
 Then, for each $t<T(\vartheta,\vartheta_0)$,
 $k_t^{\Lambda,N} := Q^\sigma_{\vartheta\vartheta_0}(t) q_0^{\Lambda,N}$ has the positivity property as in (\ref{19}),
 and hence $k_t^{\Lambda,N}(\eta)\geq 0$ for $\lambda$-almost all $\eta$.
\end{corollary}
\begin{proof}
 By (\ref{53}) it follows that $S^\dag R^{\Lambda,N}_0\in \mathcal{G}_0^{+}$ as $S^\dag$ is stochastic.
 Then the proof follows by (\ref{7z}), cf. (\ref{55}).
\end{proof}
In the sequel, we use one more property of
$Q^\sigma_{\vartheta\vartheta_0}(t)$. Recall that $W_x$ us defined
in (\ref{24}).
\begin{lemma}
 \label{F4lm}
Let $q_0^{\Lambda,N}$ and $\vartheta$, $\vartheta_0$, $t$ be as in
Lemma \ref{F3lm} and Corollary \ref{Fco}. Then, for each $x\in
\mathds{R}^d$ and $G\in B_{\rm bs}^{\star}(\Gamma_0)$, it follows
that
\begin{equation}
 \label{69}
 \langle \! \langle G, k^{\Lambda,N}_t \rangle \!\rangle \geq \langle \! \langle G, e(\tau_x;\cdot) W_x k^{\Lambda,N}_t \rangle \!\rangle.
\end{equation}
Hence
\begin{equation}
  \label{69z}
k^{\Lambda,N}_t (\eta) \geq  e(\tau_x;\eta) (W_x k^{\Lambda,N}_t )(\eta),
\end{equation}
holding for all $x\in \mathds{R}^d$ and $\lambda$-almost all
$\eta\in \Gamma_0$.
\end{lemma}
\begin{proof}
Until the end of this proof $t$, $x$ and all other variables are fixed. Then by (\ref{24}) and (\ref{Minlos}) we have
\begin{eqnarray}
 \label{69a}
 {\rm RHS(\ref{69})}& = & \int_{\Gamma_0} G(\eta) e(\tau_x;\eta)
 \left(\int_{\Gamma_0} e(t_x;\xi) k_t^{\Lambda,N} (\eta\cup \xi) \lambda ( d\xi)\right) \lambda(d\eta) \\[.2cm] \nonumber
 & = & \int_{\Gamma_0} k_t^{\Lambda,N} (\eta) \left( \sum_{\xi\subset \eta}e(t_x;\xi)e(\tau_x; \eta\setminus \xi) G (\eta\setminus \xi) \right)
 \lambda ( d\eta)  \\[.2cm] \nonumber
 & = & \int_{\Gamma_0} V_x (\eta)(t) R^{\Lambda, N}_t (\eta) \lambda ( d\eta),
\end{eqnarray}
where the last equality follows by (\ref{61}), $R^{\Lambda, N}_t = S^\dag (t) R^{\Lambda, N}_0$ and
\begin{gather*}
 V_x (\eta) = \sum_{\zeta \subset \eta} \sum_{\xi\subset \zeta} e(t_x;\xi) e(\tau_x; \zeta \setminus \xi) G (\zeta \setminus \xi).
\end{gather*}
Proceeding as in the proof of (\ref{A6}) (see Appendix below) we show that
\begin{equation*}
 V_x (\eta) = e(\tau_x;\eta) (KG)(\eta).
\end{equation*}
We use this in (\ref{69a}) to get that
\begin{equation*}
 {\rm LHS(\ref{69})} - {\rm RHS(\ref{69})} =  \langle \! \langle K G, [1- e(\tau_x;\cdot)]R^{\Lambda,N}_t \rangle \!\rangle \geq 0,
\end{equation*}
which yields (\ref{69}). The latter inequality follows by: (a) the
positivity of $KG$, see (\ref{7z}); (b) the positivity of $1-
e(\tau_x;\eta)$, see (\ref{24}); (c) the fact that $R^{\Lambda,N}_t
= S^\dag (t) R^{\Lambda,N}_0$, $S^\dag$ being a stochastic semigroup
and (\ref{53}). The proof of (\ref{69z}) is obvious.
 \end{proof}

\subsubsection{Taking the limits}

By Corollary \ref{Fco} we have that $k_t^{\Lambda,N}$ has the positivity in question. Thus, we first fix $\sigma>0$ and pass to the limit
$\Lambda\to \mathds{R}^d$ and $N\to +\infty$. A sequence of compact subsets $\{\Lambda_n\}_{n\in \mathds{N}}$ is said to be
cofinal if: (i) $\Lambda_n \subset \Lambda_{n+1}$ for each $n$; (ii) each $x\in \mathds{R}^d$ is eventually contained in its members.
The proof of the next statement in a sense repeats the proof of \cite[Lemma 5.9]{Krzys}
\begin{lemma}
 \label{G1lm}
 Let $\mu_0$, $\vartheta$, $\vartheta_0$ and $t$ be as in Lemma \ref{F3lm}. Then, for any cofinal sequence $\{\Lambda_n\}_{n\in \mathds{N}}$
 and $G\in B_{\rm bs}(\Gamma_0)$, we have that
 \begin{equation*}
  \lim_{n \rightarrow \infty}\left(
\lim_{N \rightarrow \infty}\langle \!\langle G, k_t^{\Lambda_n,
N}\rangle \!\rangle\right) = \langle \!\langle G,
k_t^{\sigma}\rangle \!\rangle, \quad k_t^{\sigma} := Q^\sigma_{\vartheta\vartheta_0}(t) k_{\mu_0}.
 \end{equation*}
\end{lemma}
\begin{proof}
For $ k_t^{\Lambda_n,
N} =  Q^\sigma_{\vartheta\vartheta_0} (t) q_0^{\Lambda_n,
N}$, by (\ref{50}) we have
 \begin{equation*}
\langle \!\langle G, k_t^{\Lambda_n,N}\rangle \!\rangle = \langle
\!\langle H^\sigma_{\vartheta_0 \vartheta}(t) G, q_0^{\Lambda_n,
N}\rangle \!\rangle, \quad  \langle \!\langle G, k_t^{\sigma}\rangle
\!\rangle = \langle \!\langle H^\sigma_{\vartheta_0 \vartheta}(t) G,
k_{\mu_0}\rangle \!\rangle.
 \end{equation*}
Set $G_t =  H^\sigma_{\vartheta_0 \vartheta}(t) G$ and then write
\begin{equation*}
  \delta(n,N)= \langle\!\langle G_t, k_{\mu_0} \rangle \!\rangle -\langle\!\langle G_t, q_0^{\Lambda_n, N}\rangle
  \!\rangle = \langle\!\langle G_t, k_{\mu_0} - q_0^{\Lambda_n, N} \rangle
  \!\rangle =:J^{(1)}_n + J^{(2)}_{n,N}.
\end{equation*}
Here
\begin{eqnarray}
  \label{76}
J^{(1)}_n &= & \int_{\Gamma_0} G_t(\eta) k_{\mu_0}(\eta) \left(1 -
\mathds{1}_{\Gamma_{\Lambda_n}}(\eta)\right) \lambda(d\eta),
    \\[.2cm] \nonumber
    J^{(2)}_{n,N} &= & \int_{\Gamma_0} G_t(\eta) \left( k_{\mu_0}(\eta) \mathds{1}_{\Gamma_{\Lambda_n}}(\eta) - q_0^{\Lambda_n, N}(\eta) \right)
    \lambda(d\eta),
\end{eqnarray}
and $\mathds{1}_{\Gamma_{\Lambda_n}}$ is the indicator of
$\Gamma_{\Lambda_n}$. To prove that, for each $n$, $J^{(2)}_{n,N}\to 0$
as $N\to +\infty$, we rewrite it as follows
\begin{eqnarray}
\label{77} J^{(2)}_{n,N} &= & \int_{\Gamma_0} \Big[ G_t(\eta)
\int_{\Gamma_{\Lambda_n}} R_{\mu_0}^{\Lambda_n}(\eta \cup \xi)
\mathds{1}_{\Gamma_{\Lambda_n}}(\eta) \left( 1 -  I_N(\eta \cup \xi)
\right) \lambda(d\xi) \Big] \lambda(d\eta)
    \\[.2cm] \nonumber
&= & \int_{\Gamma_0} G_t(\eta) \int_{\Gamma_0}
R_{\mu_0}^{\Lambda_n}(\eta \cup \xi)
\mathds{1}_{\Gamma_{\Lambda_n}}(\eta \cup \xi)\left(1 - I_N(\eta
\cup \xi) \right) \lambda(d\xi) \lambda(d\eta)
    \\[.2cm] \nonumber
&=& \int_{\Gamma_{\Lambda_n}} \sum_{\xi \subset \eta} G_t(\xi)
R_{\mu_0}^{\Lambda_n}(\eta) \left( 1 - I_N(\eta) \right) \lambda(d\eta)
    \\[.2cm] \nonumber
&= & \sum_{m = N+1}^\infty \frac{1}{m!} \int_{(\Lambda_n)^m}
R_{\mu_0}^{\Lambda_n}(\{x_1, \ldots, x_m\})\\[.2cm] \nonumber & \times & \sum_{k = 0}^m \ \sum_{\{i_1,
\ldots i_k \} \subset \{1, \ldots m \}} G_t^{(k)}(x_{i_1}, \ldots,
x_{i_k}) dx_1 \cdots dx_m,
\end{eqnarray}
where $I_N(\eta)=1$ whenever $|\eta|\leq N$ and $I_N(\eta)=0$
otherwise. Recall that we assume $k_{\mu_0} \in \mathcal{K}_{\vartheta_0}$. Then by (\ref{52a}) and (\ref{22})
it follows that
\[
 R_{\mu_0}^{\Lambda_n}(\{x_1, \ldots, x_m\})  \leq k_{\mu_0}(\{x_1, \ldots, x_m\}) \leq \|k_{\mu_0}\|_{\vartheta_0} e^{\vartheta_0
 m},
\]
holding for all $m$.  By means of this estimate we obtain from
(\ref{77}) the following
\begin{eqnarray}
\label{78}
& & |J^{(2)}_{n,N}|  \leq  \|k_0\|_{\vartheta_0} \\[.2cm]\nonumber &   & \quad \times  \sum_{m =
N+1}^\infty \frac{1}{m!} e^{\vartheta_0 m} \int_{(\Lambda_n)^m} \sum_{k
= 0}^m \ \sum_{\{i_1, \ldots i_k \} \subset \{1, \ldots m \}}
|G_t^{(k)}(x_{i_1}, \ldots, x_{i_k})| dx_1 \cdots dx_m
    \\[.2cm] \nonumber
& & \quad \leq  \|k_0\|_{\vartheta_0} \sum_{m = N+1}^\infty \frac{1}{m!}
e^{\vartheta_0 m} \sum_{k = 0}^m \frac{m!}{k! (m-k)!}
\|G_t^{(k)}\|_{L^1\left((\mathds{R}^d)^k\right)} |\Lambda_n|^{m-k},
\end{eqnarray}
where $|\Lambda|$ stands for the Lebesgue measure of $\Lambda$. The
sum over $m$ in the last line of (\ref{78}) is the remainder of the
series
\begin{eqnarray*}
& & \sum_{m = 0}^\infty \sum_{k = 0}^m  \frac{e^{\vartheta_0 k}}{k!
} \|G_t^{(k)}\|_{L^1\left((\mathds{R}^d)^k\right)}
\frac{e^{\vartheta_0 (m - k)}}{(m-k)!} |\Lambda_n|^{m-k}
    \\[.2cm]
  & &    = \sum_{k = 0}^\infty  \frac{e^{\vartheta_0 k}}{k! } \|G_t^{(k)}\|_{L^1\left((\mathds{R}^d)^k\right)}
  \sum_{m = 0}^\infty  \frac{e^{\vartheta_0 m}}{m!} |\Lambda_n|^{m}
     = |G_t|_{\vartheta_0} \exp\Big( e^{\vartheta_0} |\Lambda_n| \Big).
\end{eqnarray*}
Hence, by (\ref{78}) we obtain that
\begin{equation*}
  \delta_n := \lim_{N\to +\infty} \delta (n,N) = J_n^{(1)}.
\end{equation*}
Then it remains to show that
$\delta_n\to 0$. By (\ref{76}) we have
\[
|J^{(1)}_n| \leq \sum_{p=1}^\infty \frac{1}{p!}
\int_{(\mathds{R}^d)^p} |G_t^{(p)} (x_1, \ldots, x_p)| k_{\mu_0}^{(p)}(x_1,
\ldots, x_p) \sum_{l=1}^p \mathds{1}_{\Lambda_N^c}(x_l) dx_1 \cdots
dx_p.
\]
Since $k_{\mu_0} \in \mathcal{K}_{\vartheta_0}$ and $G_t^{(p)}$ and $k_{\mu_0}^{(p)}$
are symmetric, we rewrite the above estimate as
\[
|J^{(1)}_n| \leq \|k_{\mu_0}\|_{\vartheta_0} \sum_{p=1}^\infty \frac{p}{p!}
e^{\vartheta_0 p} \int_{\Lambda_n^c}\int_{(\mathds{R}^d)^{p-1}}
|G_t^{(p)} (x_1, \ldots, x_p)| dx_1 \cdots dx_p.
\]
For each $t<T(\vartheta,\vartheta_0)$, one finds $\epsilon>0$ such that
$t<T(\vartheta+\epsilon,\vartheta_0+\epsilon)$, see (\ref{33}). We fix
these $t$ and $\epsilon$. Since $G_0 \in B_{\rm bs}(\Gamma_0)$, and hence $G_0 \in
\mathcal{G}_{\vartheta_*+\epsilon}$, by (\ref{49}) and (\ref{50}) we have that
$G_t \in \mathcal{G}_{\vartheta_0+\epsilon}$. We apply this in the
estimate above and obtain
\begin{eqnarray}
  \label{79}
|J^{(1)}_n|  & \leq &  \frac{\|k_0\|_{\vartheta_0}}{e \epsilon }
\Delta_n,\\[.2cm] \nonumber
\Delta_n & := & \sum_{p=1}^\infty \frac{1}{p!} e^{(\vartheta_0+\epsilon)
p} \int_{\Lambda_n^c}\int_{(\mathds{R}^d)^{p-1}} |G_t^{(p)}
(x_1, \ldots, x_p)| dx_1 \cdots dx_p.
\end{eqnarray}
Furthermore, for each $M\in \mathds{N}$, we have
\begin{gather}
  \label{A51}
  \Delta_n \leq \Delta_{n,M}^{(1)} + \Delta_{M}^{(2)},\\[.2cm]
  \nonumber
 \Delta_{n,M}^{(1)} :=  \sum_{p=1}^M \frac{1}{p!} e^{(\vartheta_0+\epsilon)
p} \int_{\Lambda_n^c}\int_{(\mathds{R}^d)^{p-1}} |G_t^{(p)}
(x_1, \ldots, x_p)| dx_1 \cdots dx_p, \\[.2cm] \nonumber
\Delta_{M}^{(2)} := \sum_{p=M+1}^\infty \frac{1}{p!}
e^{(\vartheta_0+\epsilon) p} \int_{(\mathds{R}^d)^{p}} |G_t^{(p)}
(x_1, \ldots, x_p)| dx_1 \cdots dx_p.
\end{gather}
Fix some $\varepsilon>0$, and then pick $M$ such that
$\Delta^{(2)}_M < \varepsilon/2$, which is possible since
\[
\sum_{p=1}^\infty \frac{1}{p!} e^{(\vartheta_0+\epsilon) p}
\int_{(\mathds{R}^d)^{p}} |G_t^{(p)} (x_1, \ldots, x_p)| dx_1 \cdots
dx_p = |G_t|_{\vartheta_0+\epsilon},
\]
as $G_t \in \mathcal{G}_{\vartheta_0+\epsilon}$. At the same time,
\[
\forall p\in \mathds{N} \qquad g_p (x) :=
\int_{(\mathds{R}^d)^{p-1}}|G^{(p)}_t (x, x_2, \dots , x_p)|d x_2
\cdots dx_p \in L^1(\mathds{R}^d).
\]
Since the sequence $\{\Lambda_n\}$ is exhausting, for $M$
satisfying $\Delta^{(2)}_M < \varepsilon/2$, there exists $n_1$ such
that, for $n
> n_1$, the following holds
\[
\Delta_{n,M}^{(1)}=  \sum_{p=1}^M\frac{1}{p!} e^{(\vartheta_0+\epsilon)
p} \int_{\Lambda_n^c} g_p(x) dx < \frac{\varepsilon}{2}.
\]
By (\ref{A51}) this yields $\Delta_n<\varepsilon$ for all $n>n_1$,
which by (\ref{79}) completes the proof.
\end{proof}
Our next aim is to prove (\ref{42}).
\begin{lemma}
 \label{G2lm}
 Let $\mu_0$, $\vartheta$, $\vartheta_0$ be as in Lemma \ref{G1lm}. Then for each $G\in B_{\rm bs}(\Gamma_0)$
 and $t<T(\vartheta,\vartheta_0)/2$, it follows that
 \begin{equation}
  \label{80}
 \langle \! \langle G,Q^\sigma_{\vartheta\vartheta_0}(t) k_{\mu_0} \rangle \! \rangle \to
 \langle \! \langle G,Q_{\vartheta\vartheta_0}(t) k_{\mu_0} \rangle \! \rangle, \quad {\rm as} \ \sigma\to 0^+.
 \end{equation}
\end{lemma}
\begin{proof}
For each $\vartheta_1, \vartheta_2 \in (\vartheta_0, \vartheta)$, $\vartheta_2 > \vartheta_1$, we may write
\begin{eqnarray}
 \label{81}
 \left[Q_{\vartheta\vartheta_0}(t) - Q^\sigma_{\vartheta\vartheta_0}(t) \right]k_{\mu_0}& = & \int_0^t \frac{d}{ds}
 \left[Q^\sigma_{\vartheta\vartheta_1}(t-s) Q_{\vartheta_1\vartheta_0} (s) \right]k_{\mu_0} ds\\[.2cm] \nonumber
 & = & \int_0^t Q^\sigma_{\vartheta\vartheta_2}(t-s) D^\sigma_{\vartheta_2\vartheta_1} k_s ds,
\end{eqnarray}
where $k_s= Q_{\vartheta_1\vartheta_0} (s)k_{\mu_0}$ (see
(\ref{39})), $D^\sigma$ is equal to $L^\Delta$ (see (\ref{24})) with
$b(x)$ replaced by $(1-\psi_\sigma(x)) b(x)$ and $t$ satisfies
\begin{equation}
 \label{82a}
 t< \min\{T(\vartheta, \vartheta_2); T(\vartheta_1, \vartheta_0)\}.
\end{equation}
Now we take $G\in B_{\rm bs}(\Gamma_0)\subset \mathcal{G}_\vartheta$
and obtain by (\ref{81}) the following
\begin{equation}
  \label{82}
 \Upsilon_\sigma (t) := \langle \! \langle G,Q_{\vartheta\vartheta_0}(t) k_{\mu_0} \rangle \!
 \rangle - \langle \! \langle G,Q^\sigma_{\vartheta\vartheta_0}(t) k_{\mu_0} \rangle \!
 \rangle= \int_0^t \langle \! \langle G_{t-s}, D^\sigma_{\vartheta_2\vartheta_1}
 k_s \rangle \!
 \rangle d s,
\end{equation}
where $G_{t-s} = H^\sigma_{\vartheta_2 \vartheta}(t-s) G$, see
(\ref{50}).  Then by (\ref{24}), (\ref{25}) and (\ref{Minlos}) we
rewrite (\ref{82}) as follows
\begin{eqnarray}
  \label{83}
\Upsilon_\sigma (t)& = &  \int_{\mathds{R}^d}(1-\psi_\sigma (x))
b(x) g_t (x) d x , \\[.2cm] \nonumber
g_t (x) &:=& \int_0^t\int_{\Gamma_0} G_{t-s} (\eta\cup x)
e(\tau_x;\eta)( W_x k_s) (\eta) d s\lambda (d\eta).
\end{eqnarray}
Then by Lebesgue's dominated convergence theorem it follows that the
proof of (\ref{80}) will be done if we show that: (a) for each
$t<T(\vartheta, \vartheta_0)/2$, it is possible to choose
$\vartheta_2$ and $\vartheta_1$ in such a way that $t$ also
satisfies (\ref{82a}); (b) for these $t$, $\vartheta_2$ and
$\vartheta_1$, $g_t$ defined in the second line of (\ref{83}) is
absolutely integrable in $x$. To get (a) we proceed as in
\cite{Krzys}, see the very end of the proof of Lemma 5.2 in that
paper. Set $\vartheta_1 = (\vartheta + \vartheta_0)/2$ and
$\vartheta_2 = \vartheta_1 + \varepsilon \beta
(\vartheta)e^{-\vartheta}$ with $\varepsilon>0$ satisfying
$\vartheta_2 < \vartheta$. For this choice, by (\ref{33}) we have
\[
T(\vartheta_1, \vartheta_0) = \frac{\vartheta - \vartheta_0}{2 \beta
(\vartheta_1)} e^{\vartheta_0}\geq \frac{1}{2} T(\vartheta,
\vartheta_0)
>t,
\]
since $\beta$ is an increasing function, see (\ref{29}). At the same
time,
\[
T(\vartheta, \vartheta_2) = \frac{\vartheta - \vartheta_1}{\beta
(\vartheta)}  e^{\vartheta_2}- \varepsilon
e^{-(\vartheta-\vartheta_2)} > \frac{1}{2} T(\vartheta, \vartheta_0)
- \varepsilon.
\]
As $t$ is fixed, one can take positive $\varepsilon < \frac{1}{2}
T(\vartheta, \vartheta_0) - t$, that yields (a). To get (b) we first
use (\ref{25}) and then (\ref{37b}), which for $s\leq t$ and $t$
satisfying (\ref{82a}) yields
\begin{eqnarray*}
|W_x(k_s) (\eta)| & \leq & \|k_s\|_{\vartheta_1}
e^{\vartheta_1|\eta|} \exp\left(\langle \phi \rangle e^{\vartheta_1}
\right) \leq c(t, k_{\mu_0}) e^{\vartheta_1|\eta|}, \\[.2cm] \nonumber
c(t, k_{\mu_0}) & := & \frac{T(\vartheta_1 , \vartheta_0)
\|k_{\mu_0}\|_{\vartheta_0}}{T(\vartheta_1 , \vartheta_0) - t}
\exp\left(\langle \phi \rangle e^{\vartheta} \right).
\end{eqnarray*}
We use this estimate to get the following
\begin{equation}
 \label{85}
\int_{\mathds{R}^d} |g_t(x) | d x  \leq  c(t,k_{\mu_0})  \int_0^t
h(s) ds ,
\end{equation}
where
\begin{eqnarray*}
h (s) & = & \int_{\mathds{R}^d} \int_{\Gamma_0} |G_s (\eta \cup x)
e^{\vartheta_1 |\eta|} d x \lambda (d \eta) \\[.2cm] & = & e^{-\vartheta_1}
\int_{\Gamma_0} |\eta| e^{-(\vartheta_2 -\vartheta_1) |\eta|} |G_s
(\eta)|e^{\vartheta_2 |\eta|} \lambda ( d \eta) \\[.2cm] & \leq &
\frac{|G_s|_{\vartheta_2}}{e^{1+\vartheta_1}(\vartheta_2 -
\vartheta_1)} \leq \frac{T(\vartheta,
\vartheta_2)|G|_{\vartheta}}{(T(\vartheta, \vartheta_2)
-t)e^{1+\vartheta_1}(\vartheta_2 - \vartheta_1)},
\end{eqnarray*}
where we used the fact that the operator norm of
$H^\sigma_{\vartheta_2 \vartheta}(s)$ can be estimated according to
(\ref{37b}). Now we apply the latter estimate in (\ref{85}) to get
the integrability in question, that completes the proof.
\end{proof}
\subsubsection{Proof of Lemma \ref{F2lm}} To prove the lemma we have
to show that $Q_{\vartheta_2\vartheta_1} k \in
\mathcal{K}^\star_{\vartheta_2}$ whenever $k \in
\mathcal{K}^\star_{\vartheta_1}$. By (\ref{24}) and (\ref{15}) we
conclude that
\[
\frac{d}{dt} k_t(\varnothing) = (L^\Delta k_t)(\varnothing) =0,
\]
which means that $(Q_{\vartheta_2\vartheta_1} (t) k)(\varnothing ) =
k(\varnothing) = 1$. Then it remains to prove that
$Q_{\vartheta_2\vartheta_1} (t) k$ satisfies (\ref{19}) with an
arbitrary $G\in B_{\rm bs}^\star(\Gamma_0)$, which holds by
Corollary \ref{Fco} and then by Lemmas \ref{G1lm} and \ref{G2lm}.

\subsection{The proof of Theorem \ref{1tm}}

By Lemma \ref{F2lm} $k_t = Q_{\vartheta \vartheta_0}(t) k_{\mu_0}$, $t< T(\vartheta,\vartheta_0)/2$  satisfies (\ref{15}) and is the correlation function of a unique
state $\mu_t \in \mathcal{P}_{\rm exp}(\Gamma)$. Thus, to prove Theorem \ref{1tm} we have to construct the continuation of this $k_t$ to
all $t>0$ that satisfies the estimate stated in (i) and the equation (\ref{15}) in the form stated in (ii).
First we show that $k_t$ with $t<T(\vartheta, \vartheta_0)/2$ satisfies the estimate in question. To this end we use the correlation function $k_t^0$ of the
state $\mu_t^0$ corresponding to the free immigration case described by (\ref{17}) with $\phi=0$. This correlation function can be found explicitly, see
\cite[eq. (2.21)]{KK2}. In our case, it is
\begin{equation}
 \label{86}
 k^0_t (\eta) = \sum_{\xi \subset \eta} e (\varphi_t;\xi) k_{\mu_0} (\eta\setminus \xi) , \qquad
\varphi_t(x) = b(x) t.
\end{equation}
By direct inspection one gets that $k^0_t$ satisfies the equation
\begin{equation}
 \label{87}
 \frac{d}{dt} k^0_t (\eta) = (L^{0,\Delta}k^0_t)(\eta) = \sum_{x\in \eta} b(x) k_t^0 (\eta\setminus x),
\end{equation}
as well as the estimate stated in (i).
\begin{lemma}
 \label{C1lm}
For all $t<T(\vartheta, \vartheta_0)/2$ and $\lambda$-almost all $\eta\in \Gamma_0$, $k_t =Q_{\vartheta\vartheta_0}(t) k_{\mu_0}$ satisfies
\begin{equation}
 \label{88}
 k_t (\eta) \leq k_t^0(\eta).
\end{equation}
\end{lemma}
\begin{proof}
By (\ref{87}) one concludes that $L^{0,\Delta}$ satisfies the
estimate as in (\ref{29}) with $\beta(\vartheta)$ replaced by
$\beta^0(\vartheta)= \bar{b}\leq \beta (\vartheta)$. This can be
used to define bounded operators
$(L^{0,\Delta})^n_{\vartheta_2\vartheta_1}$ satisfying (\ref{36}),
(\ref{36a}) and (\ref{36b}), and then
$Q^{0}_{\vartheta_2\vartheta_1} (t)$, $t<
T^0(\vartheta_2,\vartheta_1)$, defined by (\ref{37}) with the use of
these $(L^{0,\Delta})^n_{\vartheta_2\vartheta_1}$. Here
\begin{equation}
  \label{88a}
  T^0(\vartheta_2,\vartheta_1) :=
  \frac{\vartheta_2-\vartheta_1}{\bar{b}}e^{-\vartheta_1}.
\end{equation}
Note that $T^0(\vartheta_2,\vartheta_1) >
T(\vartheta_2,\vartheta_1)$, see (\ref{33}). Then
\begin{equation*}
k_t^0 = Q^0_{\vartheta \vartheta_0}(t)k_{\mu_0} \in
\mathcal{K}_\vartheta, \qquad t< T^0(\vartheta,\vartheta_0).
\end{equation*}
On the other hand, by (\ref{86}) we have that $k^0_t \in
\mathcal{K}_{\vartheta_t}$ with $\vartheta_t = \log
(e^{\vartheta_0}+\bar{b}t)$. By (\ref{88a}) one can show that $t<
T^0(\vartheta,\vartheta_0)$ implies $\vartheta_t < \vartheta$, which
means that $k_t$ lies in a smaller member of the scale
$\{\mathcal{K}_\vartheta\}$ than $\mathcal{K}_\vartheta$.

To prove that (\ref{88}) holds we fix $t<T(\vartheta,\vartheta_0)/2$
and then similarly as in (\ref{81}) write
\begin{equation}
  \label{90}
  k^0_t - k_t = Q^0_{\vartheta \vartheta_0}(t)k_{\mu_0} - Q_{\vartheta
  \vartheta_0}(t)k_{\mu_0} = \int_0^t Q^0_{\vartheta
  \vartheta_2}(t-s)
  D_{\vartheta_2 \vartheta_1} Q_{\vartheta_1
  \vartheta_0}(s)k_{\mu_0} ds,
\end{equation}
where $\vartheta_2$ and $\vartheta_1< \vartheta_2$ are chosen in
$(\vartheta_0 , \vartheta)$ in such a way that $t$ satisfies
(\ref{82a}), see the proof of Lemma \ref{G2lm}. Moreover, $D
 = L^{0,\Delta}
- L^\Delta$ and hence $D$ is a positive operator as follows by Lemma
\ref{F4lm}, see also (\ref{87}) and (\ref{24}). By (\ref{87})
$L^{0,\Delta}$ is positive; hence, so is also $Q^0_{\vartheta
  \vartheta_2}(t-s)$. At the same time, by Lemma \ref{F2lm} $Q_{\vartheta_1
  \vartheta_0}(s)k_{\mu_0} \in \mathcal{K}_{\vartheta_1}^\star$ and
thus is also positive, cf. Corollary \ref{Fco}. Then the right-hand
side of (\ref{90}) is positive, that yields (\ref{88}).
\end{proof}
\begin{corollary}
  \label{C1co}
For each $t<T(\vartheta, \vartheta_0)/2$, we have that $k_t =
Q_{\vartheta \vartheta_0}(t)k_{\mu_0} \in
\mathcal{K}_{\vartheta_t}$.
\end{corollary}
\begin{proof}
The proof follows by (\ref{88}) and the fact that $k_t^0 \in
\mathcal{K}_{\vartheta}$ for these values of $t$.
\end{proof}
The value of $\vartheta_0$ in the latter statements is predetermined
by the choice of the initial state $\mu_0$. At the same time, the
choice of $\vartheta$ can be optimized to make the time interval in
these statements as long as possible. By simple calculations we then
get that
\begin{equation}
  \label{91}
\sup_{\vartheta > \vartheta_0} T(\vartheta, \vartheta_0) =
\tau(\vartheta_0) := \frac{\delta}{\bar{b}}\exp\left( \vartheta_0 -
\frac{1}{\delta}\right),
\end{equation}
where positive $\delta=\delta (\vartheta_0)$ is the unique solution
of the equation
\begin{equation}
  \label{92}
  \delta e^\delta = \exp\left( - \vartheta_0 - \log \langle \phi \rangle
  \right).
\end{equation}
Note that the latter implies
\begin{equation}
 \label{92z}
 \tau(\vartheta) = \frac{1}{\bar{b}\langle \phi \rangle} \exp\left(- \delta (\vartheta) - \frac{1}{\delta (\vartheta)}   \right).
\end{equation}
{\it Proof of Theorem \ref{1tm}.} By (\ref{38}) one readily obtains
that the family $\{Q_{\vartheta_2 \vartheta_1}(t): \vartheta_1 \in
\mathds{R}, \ \vartheta_2 > \vartheta_1, \ t < T(\vartheta_2,
\vartheta_1)\}$ has the property
\begin{equation}
  \label{93}
  Q_{\vartheta_2 \vartheta_1}(t+s) = Q_{\vartheta_2 \vartheta'}(t) Q_{\vartheta'
  \vartheta_1}(s), \qquad \vartheta'\in (\vartheta_1, \vartheta_2),
\end{equation}
where positive $s$ and $t$ ought to satisfy: $t< T(\vartheta',
  \vartheta_1)$, $s< T(\vartheta_2,
  \vartheta')$, $t+ s < T(\vartheta_2,
  \vartheta_1)$.

Take some $\epsilon \in (0,1/2)$ and set $s_1 = \epsilon
\tau(\vartheta_0)$, see (\ref{91}). For $t\leq s_1$, by Lemma
\ref{F2lm} and Corollary \ref{C1co} it follows that $k_t=
Q_{\vartheta^0\vartheta_0}(t) k_{\mu_0}$ with $\vartheta^0 :=
\vartheta_0 +\delta (\vartheta_0)$ satisfies the estimate stated in
claim (i). Recall that $\vartheta_t=\log (e^{\vartheta_0 }+ \bar{b}
t)$ and thus $\vartheta^0>\vartheta_t$ for all $t\leq s_1$. Now we
fix some $T\leq s_1$ and then take a positive $\varepsilon< T$. For
each $0<t\leq T-\varepsilon$, by the mentioned estimate and
(\ref{28}) it follows that $k_t\in
\mathcal{K}_{\vartheta_{T-\varepsilon}}$. By (\ref{38}) and then by
(\ref{30})  we have that
\begin{eqnarray}
 \label{93a}
 \frac{d}{dt} k_t & = & L^\Delta_{\vartheta^0\vartheta_{T-\varepsilon}} k_{t}
 = L^\Delta_{\vartheta_{T-\varepsilon/2}\vartheta_{T-\varepsilon}} k_{t} \\[.2cm] \nonumber
 & = & L^\Delta_{\vartheta_{T}\vartheta_{T-\varepsilon}} k_{t}= L^\Delta_{\vartheta_T} k_{t} .
 \end{eqnarray}
To prove the continuity and continuous differentiability stated in (ii) we use again (\ref{93}). Let positive $s$ be such that
$t+s\leq T- \varepsilon$ with $\varepsilon$ as in (\ref{93a}). For $\vartheta>\vartheta_{s_1}$ and $t$ as in (\ref{93a}), we have that
$t<T(\vartheta, \vartheta_0)$. Let also $s< T(\vartheta^0,\vartheta)$ with this $\vartheta$. Then
\begin{equation}
 \label{93b}
k_{t+s} = Q_{\vartheta^0\vartheta_0}(t+s) k_{\mu_0} = Q_{\vartheta^0\vartheta}(s)Q_{\vartheta\vartheta_0}(t)k_{\mu_0}= Q_{\vartheta^0\vartheta}(s)
k_t =  Q_{\vartheta_T\vartheta_{T-\varepsilon}}(s)
k_t.
\end{equation}
Here we have taken into account that $k_t\in
\mathcal{K}_{\vartheta_{T-\varepsilon}}$ and the property mentioned
in Remark \ref{Frk}. Of course, the latter equality in (\ref{93b})
makes sense for sufficiently small $s$. Then the continuity and
continuous differentiability follow by (\ref{93b}) and (\ref{38}).
Thus, both claims (i) and (ii) hold true for $t\leq s_1$.

Set $\vartheta_1^* = \vartheta_{s_1}$ and $\vartheta^1 = \vartheta_1^* + \delta (\vartheta_1^*)$. By Corollary \ref{C1co} it
follows that $k_{s_1} \in \mathcal{K}_{\vartheta_1^*}$. For each
$k\in \mathcal{K}_{\vartheta_1^*}$, we have that
$Q_{\vartheta^1\vartheta_1^*} (t) k \in \mathcal{K}_{\vartheta^1}$
for $t< \tau(\vartheta_1^*)$. Keeping this in mind we then set
\begin{equation}
  \label{94}
  k_{s_1 + t} = Q_{\vartheta^1\vartheta_1^*} (t) k_{s_1}, \qquad t\in[0,
  \tau(\vartheta_1^*)).
\end{equation}
For $t$ such that $s_1 + t < \tau(\vartheta_0)$, by (\ref{93}) and (\ref{94}) it follows
that $k_{s_1 + t} = Q_{\vartheta^1\vartheta_0}
(s_1+t)k_{\mu_0}$.  In view of the latter, we set
\begin{equation*}
  k_t = \left\{\begin{array}{ll} Q_{\vartheta_1^* \vartheta_0}(t)
  k_{\mu_0} , \qquad &{\rm for} \ t\leq s_1;\\[.3cm] Q_{\vartheta^1 \vartheta_1^*}(t-s_1)
  k_{s_1} , \qquad &{\rm for} \ t\in [s_1, s_1 +
  \tau(\vartheta_1^*)).
  \end{array} \right.
\end{equation*}
As above, we show that this $k_t$ has all the properties stated in
both claims (i) and (ii) for $t\leq s_1 + s_2$, with $s_2 :=
\epsilon \tau(\vartheta_1^*)$. Then we repeat the same procedure
again and again. That is, we set $s_n =
\epsilon\tau(\vartheta^*_{n-1})$,
\begin{equation}
 \label{96}
\vartheta_n^* = \vartheta_{s_1 + \cdots + s_n} = \log\left(e^{\vartheta_0} + \bar{b}(s_1 + \cdots + s_n)\right),
\end{equation}
and $\vartheta^n = \vartheta_n^* + \delta (\vartheta_n^*)$. Thereafter, for a given $m$ and
$t\in [s_1 + \cdots + s_m, s_1 + \cdots + s_{m+1}]$, we set
\begin{equation*}
 k_t = Q_{\vartheta^m, \vartheta_m^*}(t- (s_1 +\cdots + s_m)) k_{s_1 +\cdots + s_m},
\end{equation*}
and prove that this $k_t$ has all the properties in question. To
complete the proof we have to show that the intervals just mentioned
cover $[0,+\infty)$, which means that
\[
 \sum_{n=1}^\infty s_n = \epsilon \sum_{n=1}^\infty \tau(\vartheta_n^*) =  +\infty.
\]
Assume that $\sum_{n=1}^\infty s_n=\bar{s}<\infty$. Then, for each
$n$, $\vartheta^*_n$ satisfies $\vartheta^*_1 \leq \vartheta^*_n
\leq \bar{\vartheta}:=\log(e^{\vartheta_0} + \bar{b} \bar{s})$, see
(\ref{96}). At the same time, the assumed convergence yields
$\tau(\vartheta_n^*) \to 0$ as $n\to +\infty$, which by (\ref{92z})
and (\ref{92}) implies $|\vartheta_n^*| \to + \infty$ as $n\to
+\infty$, that contradicts the boundedness just mentioned. \hfill
$\square$

\subsection{The proof of Theorem \ref{2tm}}

For a given compact $\Lambda$, let $m_\Lambda$ be the minimal number
of the balls $\Delta_x$ that cover $\Lambda$. Then there exist
compact $\Lambda_l$, $l=1, \dots , m_\Lambda$ such that:
\begin{equation}
 \label{98}
 (a) \ \sup_{x,y\in \Lambda_l} |x-y| \leq r; \qquad (b) \ \ \Lambda \subset
 \bigcup_{l=1}^{m_\Lambda}
 \Lambda_l,
\end{equation}
and hence
\begin{equation}
  \label{98a}
  \mu_t (N_\Lambda) \leq \sum_{l=1}^{m_\Lambda} \mu_t
  (N_{\Lambda_l}).
\end{equation}
Then it is enough to prove (\ref{32b}) for some (hence, for each)
$\Lambda_l$, for which we set $F(\gamma) = e^{\phi_* N_{\Lambda_l}
(\gamma)}$, cf. (\ref{I4d}). By Theorem \ref{1tm} and Corollary
\ref{M1co} we know that the correlation function $k_{\mu_t}$ lies in
$\mathcal{K}_{\vartheta_t}$. Hence, one can calculate
$f(t):=\mu_t(F)$, see (\ref{I4d}). Note that
$f(t)=\mu^{\Lambda_l}_t(F)$, cf. (\ref{52}). Then, cf. (\ref{56}),
by (\ref{17}) we have
\begin{eqnarray}
 \label{99}
 \frac{d}{dt} f(t) & = & \int_{\Gamma_{\Lambda_l}} F(\gamma) (L^\dag R^{\Lambda_l}_t ) (\gamma) \lambda ( d \gamma) =
 \int_{\Gamma_{\Lambda_l}}(L F)(\gamma)  R^{\Lambda_l}_t  (\gamma) \lambda ( d \gamma) \\[.2cm] \nonumber
 & = &  \int_{\Gamma_{\Lambda_l}} \left( \int_{\mathds{R}^d} b(x) e( \tau_x;\gamma) [F(\gamma\cup x) - F(\gamma)]
 dx \right) R^{\Lambda_l}_t  (\gamma) \lambda ( d \gamma) \\[.2cm] \nonumber
 & = &  \int_{\Gamma_{\Lambda_l}} \left(\int_{\Lambda_l} b(x) (e^{\phi_*} -1)
 e( \tau_x;\gamma)F(\gamma) d x \right) R^{\Lambda_l}_t  (\gamma) \lambda ( d \gamma) \\[.2cm] \nonumber
& \leq & (e^{\phi_*} -1) \bar{b}  \int_{\Gamma_{\Lambda_l}} \Upsilon_{\Lambda_l}(\gamma) R^{\Lambda_l}_t  (\gamma) \lambda ( d \gamma)
\\[.2cm] \nonumber &\leq &  (e^{\phi_*} -1)\bar{b}  |\Lambda_l| \leq   (e^{\phi_*} -1)\bar{b}  \upsilon.
\end{eqnarray}
Here
\begin{gather*}
 \Upsilon_{\Lambda_l} (\gamma) :=  \int_{\Lambda_l} \exp\left(- \sum_{y\in \gamma} [\phi(x-y) - \phi_*] \right) d x \leq |\Lambda_l|,
\end{gather*}
since, for $\gamma \in \Lambda_l$, by item (a) in (\ref{98}) and item (c) of Assumption \ref{ass} we have that $\phi(x-y) \geq \phi_*$.
Also by (a) in (\ref{98}) it follows that $\Lambda_l$ is contained in a ball $\Delta_z$ with some $z$, that yields
$|\Lambda_l| \leq \upsilon$. The latter was used in the last line of (\ref{99}). Now by (\ref{99}) and (\ref{I4d}) we have
\begin{equation*}
 f(t) \leq f(0) + (e^{\phi_*} -1)\bar{b}  \upsilon t \leq C_{\Delta_0}^{\phi_*}(\mu_0) + (e^{\phi_*} -1)\bar{b}  \upsilon t, \qquad t\geq 0,
\end{equation*}
where we used the fact that all $\Delta_z$ have the same volume. Now
by Jensen's inequality we obtain
\[
 \phi_*\mu_t (N_{\Lambda_l} ) \leq \log f(t),
\]
that by (\ref{98a}) yields (\ref{32b}).

\section{The Mesoscopic Evolution}

\label{MESS}

In this section, we prove Theorems \ref{3tm} and \ref{4tm}.

\subsection{The proof of Theorem \ref{3tm}} It is
convenient to pass in (\ref{8z}) to a new unknown $u_t \in
L^\infty(\mathds{R}^d)$ defined by
\begin{equation}
  \label{c3}
u_t(x) = \langle \phi \rangle \left[\varrho_t(x) -
\varrho_0(x)\right].
\end{equation}
Then $\varrho_t$ solves (\ref{8z}) if and only if $u_t$ solves
\begin{equation}
  \label{k1}
\frac{d}{dt} u_t (x) = \hat{b}(x) \exp \left(- (\hat{\phi}\ast
u_t)(x) \right), \qquad u_t|_{t=0} = 0,
\end{equation}
that can also be rewritten in the form
\begin{equation}
  \label{k1a}
  u_t (x) =  \hat{b}(x) \int_0^t \exp \left(- (\hat{\phi}\ast
u_s)(x) \right) ds.
\end{equation}
Here
\begin{equation}
  \label{k2}
\hat{b}(x) = \langle \phi \rangle b(x) e^{-(\phi\ast \varrho_0)(x)},
\qquad \hat{\phi} (x) = \phi(x)/\langle \phi \rangle.
\end{equation}
For a given $T>0$, set
\begin{gather*}
\mathcal{U}_T= C\left([0,T] \to L^\infty (\mathds{R}^d)
\right)\\[.2cm] \nonumber
\mathcal{U}_T^{+}= \{ u\in \mathcal{U}_T:  u_t(x) \geq 0 , \ \ {\rm
for} \ {\rm all} \ t\in [0,T] \ {\rm and} \ {\rm a.a.} \ x\in
\mathds{R}^d\},
\end{gather*}
and equip $\mathcal{U}_T$ with the norm
\begin{equation*}
  \|u\|_T = \sup_{t\in [0,T]} \|u_t\|_{L^\infty(\mathds{R}^d)}.
\end{equation*}
Then we define $V: \mathcal{U}_T \to \mathcal{U}_T$ be setting
\begin{equation}
  \label{k4}
 \left(V(u) \right)_t(x) = \hat{b}(x) \int^t_0 \exp\left( - (\hat{\phi}\ast
 u_s)(x)\right) ds,
\end{equation}
and rewrite the problem in (\ref{k1a}) on the time interval $[0,T]$
in the form
\begin{equation}
  \label{k4a}
  u = V(u).
\end{equation}
Clearly,
\begin{equation*}
 V: \mathcal{U}^{+}_T \to \mathcal{V}_T^{+} =\{ u\in \mathcal{U}_T^{+}: u_0 = 0, \  \|u\|_T \leq b^{+}
 T\},
\end{equation*}
where $b^{+}$ is the same as in (\ref{c4}). By the inequality
$|e^{-\alpha} - e^{-\alpha'}|\leq |\alpha -\alpha'|$ that holds for
all $\alpha, \alpha'\geq 0$, one shows that $V$ is a contraction
whenever
\begin{equation}
  \label{k6}
  b^{+}T < 1,
\end{equation}
considered as a condition on $T$. In this case, the problem in
(\ref{k1}) has a unique solution on $[0,T]$ which is the fixed point
$u\in \mathcal{V}_T^{+}$ of $V$. For each $t\in [0,T]$, as a
positive element of $L^\infty (\mathds{R}^d)$ $u_t$ satisfies
$\omega_{-} (t) \leq u_t(x)\leq \omega_{+}(t)$ holding almost
everywhere on $\mathds{R}^d$. Here $\omega_{\pm}$ are to be
continuously differentiable functions such that $\omega_{\pm} (0)
=0$. To find them we write
\begin{equation}
  \label{k7}
  \frac{d}{dt}\omega_{-}(t) \leq \frac{d}{dt}u_t(x) \leq
  \frac{d}{dt}\omega_{-}(t),
\end{equation}
which together with the zero initial condition will yield the bounds
in question. On the other hand, for these bounds by (\ref{k1}) we
have, see (\ref{k2}) and (\ref{c4}),
\begin{equation}
  \label{k8}
b^{-} e^{-\omega_{+}(t)} \leq \frac{d}{dt}u_t(x) \leq b^{+}
e^{-\omega_{-}(t)}.
\end{equation}
Now we combine (\ref{k7}) with (\ref{k8}) and obtain that
$\omega_{\pm}$ ought to satisfy
\begin{equation}
  \label{k9}
 \frac{d}{dt}\omega_{-}(t) = b^{-} e^{-\omega_{+}(t)}, \qquad \frac{d}{dt}\omega_{+}(t) = b^{+}
 e^{-\omega_{-}(t)}.
\end{equation}
For $b_{+}=b_{-}=:b$, the solution of this system, and thereby of
(\ref{k1}), is
\begin{equation*}
  \omega_{+} (t) = \omega_{-} (t) = u_t (x)= \log \left( 1 + b t\right).
\end{equation*}
For $b_{+}>b_{-}$, by (\ref{k9}) we get
\[
\frac{d^2}{dt^2} \omega_{-}(t) = \frac{d^2}{dt^2} \omega_{+}(t),
\]
which yields
\[
\left(\frac{d}{dt} \omega_{+} (t) - \frac{d}{dt} \omega_{-} (t)
\right) = \left(\frac{d}{dt} \omega_{+} (t) - \frac{d}{dt}
\omega_{-} (t) \right)\left|_{t=0} = b^{+} - b^{-}. \right.
\]
In view of the zero initial condition, the latter yields in turn
\[
\omega_{+} (t) = \omega_{-} (t) + (b^{+} - b^{-})t.
\]
We plug this in the first equation in (\ref{k9}) that turns it into
an equation for $\omega_{-}$, the solution of which is clearly given
in the second line of (\ref{c2}). Thereafter, $\omega_{+}$ is
obtained from the formula above. Thus, with the help of (\ref{c3})
we have proved the existence of the unique solution of the kinetic
equation in (\ref{8z}) -- satisfying the bounds stated in Theorem
\ref{3tm} -- on the time interval $[0,T]$. Our aim now is to
continue it to all $t>0$. To this end, for $t>0$ we rewrite
(\ref{k1a})
\begin{eqnarray}
  \label{k11}
u_{T+t} (x) & = & \hat{b}(x) \int_0^T \exp\left( - (\hat{\phi}\ast
u_s)(x) \right) ds + \hat{b}(x) \int_0^t \exp\left( -
(\hat{\phi}\ast u_{T+s})(x) \right) ds \nonumber \\[.2cm] & = & u_T(x) + \hat{b}(x) \int_0^t \exp\left( -
(\hat{\phi}\ast u_{T+s})(x) \right) ds,
\end{eqnarray}
where we have taken into account that $u_T(x) = (V(u))_T (x)$, as
follows from the consideration above. Now we introduce
\[
u^{(1)}_t (x) = u_{T+t} (x) - u_T(x),
\]
and rewrite (\ref{k11}) in the form, cf. (\ref{k4}), (\ref{k4a}),
\begin{gather*}
u^{(1)} = V^{(1)}( u^{(1)}) , \\[.2cm] \nonumber
V^{(1)}( u^{(1)})_t(x) := \hat{b}^{(1)}(x) \int_0^t \exp\left( -
(\hat{\phi}\ast u^{(1)}_s)(x) \right) ds,\\[.2cm] \nonumber
\hat{b}^{(1)}(x) := \hat{b}(x) \exp\left( - (\hat{\phi}\ast u_T)(x)
\right).
\end{gather*}
Similarly as above, we establish that $V^{(1)}: \mathcal{U}_T^{+}
\to \mathcal{V}^{+}_{T}$ and is a contraction on
$\mathcal{V}^{+}_{T}$ with the same $T$ as in (\ref{k6}). This
yields the existence of its unique fixed point $u^{(1)} \in
\mathcal{V}^{+}_{T}$. Then the continuation on the time interval
$[T,2T]$ is
\begin{equation*}
  u_{T+t} (x) = u_T(x) + u^{(1)}_t (x).
\end{equation*}
Since $u_{t}$ on both intervals $[0,T]$ and $[T,2T]$ solves the same
differential equation, it satisfies the same bounds found from this
equation, i.e.,
\begin{equation}
  \label{k16}
\omega_{-}(t) \leq u_t(x) \leq \omega_{+} (t) , \qquad t\in [0, 2T].
\end{equation}
where $\omega_{\pm}$ are as in (\ref{c2}). The continuations of
$u_t$ beyond $[0,2T]$ are constructed by repeating the same
procedure. The bounds in (\ref{c1}) are readily obtained from
(\ref{k16}) by (\ref{c3}). This completes the proof of the theorem.

\subsection{The proof of Theorem \ref{4tm}}
In the proof of this theorem we mostly follow the line of arguments
used in the proof of Theorem 3.9 in \cite{asia1}

\subsubsection{The rescaled evolution}

Here we construct the evolution $q_{0,\varepsilon}\to
q_{t,\varepsilon}$ mentioned in the theorem. For a conceptual
background of this approach, see the corresponding parts of
\cite{asia1,FKKK} and the references therein.

Let $L^{\Delta}_{\varepsilon}$ be the operator defined in (\ref{24})
in which $\phi$ is multiplied by $\varepsilon\in (0,1]$. Next,
define the rescaling operator $R_\varepsilon$ that acts as
$(R_\varepsilon k)(\eta) = \varepsilon^{-\eta}k(\eta)$. In general,
for a correlation function, $k_\mu$,  $R_\varepsilon k_{\mu}$ need
not be the correlation function of any state. At the same time,
$R^{-1}_\varepsilon k_{\mu}$ is the correlation function of the
`thinning' $\mu_\varepsilon$ of $\mu$ defined by its Bogoliubov
functional (\ref{I1}), (\ref{I3})
\[
B_{\mu_\varepsilon}(\theta) = B_{\mu}(\varepsilon \theta)
=\int_{\Gamma} \prod_{x\in \gamma} \left( 1 + \varepsilon \theta
(x)\right) \mu(d\gamma),
\]
as the map $\theta \mapsto \varepsilon \theta$ preserves
$\varTheta$. Then we set
\begin{equation}
  \label{k17}
  L^{\varepsilon,\Delta} = R^{-1}_\varepsilon L^{\Delta}_\varepsilon
  R_\varepsilon.
\end{equation}
Next, denote, cf. (\ref{24}),
\begin{equation}
  \label{k18}
 \tau_x^\varepsilon (y) = \exp\left(- \varepsilon \phi (x-y)
 \right),  \qquad t_x^\varepsilon (y) =\frac{1}{\varepsilon}
 \left( \tau_x^\varepsilon (y) -1 \right).
\end{equation}
Note that
\begin{equation}
  \label{k19}
\tau_x^\varepsilon (y) \to 1, \qquad t_x^\varepsilon (y) \to
t^0_x(y):=- \phi(x-y), \qquad {\rm as} \ \ \varepsilon\to 0^{+},
\end{equation}
and also, cf. (\ref{10})
\begin{equation}
  \label{k20}
  \frac{1}{\varepsilon}\int_{\mathds{R}^d} \left(1 - e^{-\varepsilon \phi(x)} \right) d
  x \leq \langle \phi \rangle.
\end{equation}
Now let $W_x^\varepsilon$ be defined as in (\ref{24}) with $t_x$
replaced by $t_x^\varepsilon$. In view of (\ref{k20}),
$(W_x^\varepsilon k)(\eta)$ satisfies (\ref{25}) with the same
(hence $\varepsilon$-independent) right-hand side. This means that
$L^{\varepsilon, \Delta}k$ satisfies (\ref{26}), which allows one to
define the unbounded operators $L^{\varepsilon,
\Delta}_\vartheta=(L^{\varepsilon, \Delta}, \mathcal{D}_\vartheta)$,
$\vartheta\in \mathds{R}$ and the bounded operators $L^{\varepsilon,
\Delta}_{\vartheta\vartheta_0}$, $\vartheta> \vartheta_0$ exactly as
in the case of $\varepsilon =1$, see subsection \ref{2.3ss}.
Therefore, one can construct the family
$\{Q^\varepsilon_{\vartheta\vartheta_0}(t): \vartheta_0 \in
\mathds{R}, \ \vartheta>\vartheta_0, \ t \in [0,T(\vartheta,
\vartheta_0))\}$ with $T(\vartheta, \vartheta_0)$ as in (\ref{33}),
see also (\ref{37}). Then, for $q_{0,\varepsilon}\in
\mathcal{K}_{\vartheta_0}$, we set
\begin{equation}
  \label{k21}
q_{t,\varepsilon} = Q^\varepsilon_{\vartheta\vartheta_0}(t)
q_{0,\varepsilon}, \qquad t< T(\vartheta, \vartheta_0).
\end{equation}
This $q_{t,\varepsilon}$ satisfies, cf. (\ref{31}),
\begin{equation}
  \label{k22}
  \frac{d}{dt} q_{t,\varepsilon} = L^{\varepsilon,
\Delta}_\vartheta q_{t,\varepsilon}.
\end{equation}
By (\ref{k17}), (\ref{k18}) and (\ref{k21}) we have that
\begin{equation}
  \label{k23}
q_{t,1} = k_t = Q_{\vartheta\vartheta_0}(t) k_{\mu_0},
\end{equation}
where $k_t$ is as in Theorem \ref{1tm}.

\subsubsection{The Vlasov evolution}

Let $W^0_x$ be defined as in (\ref{24}) with $t_x$ replaced by
$t^0_x$, see (\ref{k19}). Then $W^0_x k$ also satisfies (\ref{25}).
Then the Vlasov operator is defined as, cf. (\ref{k19}),
\begin{equation}
  \label{k24}
(L^{0,\Delta} k)(\eta) = \sum_{x\in \eta} b(x) (W^0_x
k)(\eta\setminus x).
\end{equation}
Analogously as above, we define the family of operators
$\{Q^0_{\vartheta\vartheta_0}(t): \vartheta_0 \in \mathds{R}, \
\vartheta>\vartheta_0, \ t \in [0,T(\vartheta, \vartheta_0))\}$ with
$T(\vartheta, \vartheta_0)$ as in (\ref{33}). Then, for $q_{0,0}\in
\mathcal{K}_{\vartheta_0}$,  $q_{0,t} = Q^0_{\varepsilon
\varepsilon_0} (t) q_{0,0}\in \mathcal{K}_\vartheta$ is the unique
solution of the problem
\begin{equation}
  \label{k25}
  \frac{d}{dt} q_{t,0} = L^{0,
\Delta}_\vartheta q_{t,0}, \qquad q_{t,0}|_{t=0}= q_{0,0},
\end{equation}
on the time interval $[0,T(\vartheta,\vartheta_0))$.
\begin{lemma}
  \label{k1lm}
Let $q_{0,0}$ in (\ref{k25}) be the correlation function of the
Poisson state $\pi_{\varrho_0}$ with $k_{\pi_{\varrho_0}}\in
\mathcal{K}_{\vartheta_0}$. Then $q_{t,0} = Q^{0}_{\vartheta
\vartheta_0} (t) q_{0,0}$ can be continued to all $t>0$, and this
continuation is the correlation function of the Poisson state
$\pi_{\varrho_t}$ with $\varrho_t$ that solves the kinetic equation
(\ref{8z}), the existence and properties of which were established
in Theorem \ref{3tm}.
\end{lemma}
\begin{proof}
In fact, to prove this statement we have to check whether
$k_{\pi_{\varrho_t}} = e(\varrho_t;\cdot)$ satisfies the first
equality in (\ref{k25}). By (\ref{k19}) and (\ref{8}) we have that
\begin{equation*}
 W_x e(\varrho_t;\eta) = \int_{\Gamma_0} e(t^0_x;\xi) e(\varrho_t; \eta\cup\xi) \lambda ( d\xi)
 = e(\varrho_t; \eta) \exp\left(- (\phi\ast \varrho_t)(x) \right).
\end{equation*}
We plug this in (\ref{k24}) and then  get
\begin{gather*}
L^{0,\Delta} e(\varrho_t;\eta) = \sum_{x\in \eta} b(x)  \exp\left(-
(\phi\ast \varrho_t)(x) \right) e(\varrho_t; \eta\setminus x) \\[.2cm]  =
\sum_{x\in \eta} \frac{d}{dt} \varrho_t (x) e(\varrho_t;
\eta\setminus x),
\end{gather*}
see (\ref{8z}), which completes the proof.
\end{proof}

\subsubsection{The proof of Theorem \ref{4tm}}

As we assume that $\mu_0$ is Poisson approximable, see Definition
\ref{Poidf}, there exist $q_{0,\varepsilon} \in
\mathcal{K}_{\vartheta_0}$, $\varepsilon \in [0,1]$ such that: (a)
$q_{0,1} =k_{\mu_0}$; (b) $q_{0,0} =k_{\pi_{\varrho_0}} =
e(\varrho_0;\cdot)$; (c) $\| q_{0,\varepsilon} -
e(\varrho_0;\cdot)\|_{\vartheta_0} \to 0$ as $\varepsilon \to
0^{+}$. Let $q_{t,\varepsilon}$ be as in (\ref{k21}), (\ref{k22})
with this $q_{0,\varepsilon}$. By (\ref{k23}) we have that
$q_{t,1}=k_t$; hence, it remains to prove that there exist $T>0$ and
$\vartheta>\vartheta_0$ such that the convergence stated in
(\ref{rho}) does hold. For $\vartheta_0$ as just discussed, let
$\delta (\vartheta_0)$ and $\tau(\vartheta_0)$ be as in (\ref{92})
and (\ref{92z}), respectively. Then we set $\vartheta = \vartheta_0
+ \delta(\vartheta_0)$, $T = \tau(\vartheta_0)/2$ and then write
\begin{gather}
  \label{k26a}
 Q^\varepsilon_{\vartheta \vartheta_0} (t) q_{0,\varepsilon} -  Q^0_{\vartheta \vartheta_0} (t)
 e(\varrho_0;\cdot)  =  J^1_{\varepsilon} (t) + J^2_{\varepsilon}
 (t), \qquad t\leq T, \\[.2cm] \nonumber
J^1_{\varepsilon} (t)  =   \left[ Q^\varepsilon_{\vartheta
\vartheta_0} (t) - Q^0_{\vartheta \vartheta_0} (t)
 \right]  e(\varrho_0;\cdot) , \\[.2cm] \nonumber
J^2_{\varepsilon}
 (t)  =
 Q^\varepsilon_{\vartheta \vartheta_0} (t)\left(
 q_{0,\varepsilon} - e(\varrho_0;\cdot)\right).
\end{gather}
Similarly as in (\ref{81}) we write
\begin{equation}
  \label{k27}
J^1_{\varepsilon} (t) = \int_0^t Q^\varepsilon_{\vartheta
\vartheta_2} (t-s) D^\varepsilon_{\vartheta_2\vartheta_1}
e(\varrho_s;\cdot) ds
\end{equation}
with
\begin{equation}
  \label{k28}
  D^\varepsilon = L^{\varepsilon,\Delta} - L^{0,\Delta}.
\end{equation}
In the right-hand side of (\ref{k27}), we have taken into account
that $ Q^0_{\vartheta_1 \vartheta_0} (s)e(\varrho_0,\cdot) =
e(\varrho_s,\cdot)$ with $e(\varrho_s,\cdot)\in
\mathcal{K}_{\vartheta_1}$. The numbers $\vartheta_1$ and
$\vartheta_2$ will be chosen later. By (\ref{c1}) and (\ref{c2}),
for $\varrho_0 \leq e^{\vartheta_0}$,  one can show that $\varrho_t
(x) \leq  e^{\vartheta_0}+ \bar{b}t$. holding for all $t>0$ and both
cases $b^{+} >b^{-}$ and $b^{+} =b^{-}$. Then in (\ref{k27}) we have
that $e(\varrho_s;\cdot) \in \mathcal{K}_{\vartheta_T}$, that holds
for all $s\leq t\leq T$. Thus, we set $\vartheta_1 = \vartheta_T$.
In view of the continuity of $T(\vartheta',\vartheta)$ in both
arguments, see (\ref{33}), for each fixed $t\leq T$, one can pick
$\vartheta_2 \in (\vartheta_T, \vartheta)$ in such a way that $t-s <
T(\vartheta, \vartheta_2)$, holding for all $s\in [0,t]$, whenever
the following is satisfied
\begin{equation}
  \label{k29}
  T < T(\vartheta, \vartheta_T),
\end{equation}
for the choice $T = \tau(\vartheta_0)/2$ made above. In Appendix
below we prove that (\ref{k29}) does hold. Then we fix $t\leq T$,
pick $\vartheta_2$ as mentioned above, and then use the estimate
\[
\| Q^\varepsilon_{\vartheta \vartheta_2} (t-s)\|\leq
\frac{T(\vartheta, \vartheta_2)}{T(\vartheta, \vartheta_2)-t}
\]
where we employed (\ref{37b}) as $L^{\varepsilon,\Delta}$ satisfies
(\ref{36}) with the same right-hand side. Now we take into account
that $\|e(\varrho_s;\cdots)\|_{\vartheta_T} \leq 1$ and apply this
and the latter estimate in (\ref{k27}). This yields
\begin{equation}
  \label{k30}
  \|J^1_\varepsilon (t)\|_\vartheta \leq \frac{t T(\vartheta, \vartheta_2)}{T(\vartheta,
  \vartheta_2)-t} \|D^\varepsilon_{\vartheta_2\vartheta_T}\|.
\end{equation}
In view of (\ref{k28}), to estimate
$\|D^\varepsilon_{\vartheta_2\vartheta_T}\|$ we have to consider
$W^\varepsilon_x - W^\varepsilon_0$. By means of the standard
inequality
\[
\left|\prod_{x\in \xi} a_x - \prod_{x\in \xi} b_x \right| \leq
\sum_{x\in \xi} |a_x  - b_x| \prod_{y \in \xi\setminus x}
\max\{|a_x|;|b_xx|\}, \qquad a_x, b_x \in \mathds{R},
\]
and by (\ref{22}) for $k\in \mathcal{K}_{\vartheta_T}$ we get, see
(\ref{24}) and (\ref{k18}) and (\ref{k19}),
\begin{eqnarray*}
  \left|[(W^\varepsilon_x - W^0_x) k](\eta) \right| & \leq &
  \|k\|_{\vartheta_T} e^{\vartheta_T|\eta|} \int_{\Gamma_0} \left|
  \prod_{y\in \xi} t^\varepsilon_x (y) - \prod_{y\in \xi} t^0_x (y)
  \right| \exp\left(\vartheta_T |\xi|\right) \lambda (
  d\xi)\\[.2cm] \nonumber
 & \leq & \varepsilon\bar{\phi} \|k\|_{\vartheta_T} e^{\vartheta_T|\eta|}
  \int_{\Gamma_0} |\xi| e(\phi(x-\cdot);\xi)  \exp\left(\vartheta_T |\xi|\right) \lambda (
  d\xi)\\[.2cm] \nonumber& = &\varepsilon\bar{\phi}\langle \phi \rangle  \|k\|_{\vartheta_T}
  e^{\vartheta_T|\eta| +\vartheta_T} \exp\left( e^{\vartheta_T}\langle \phi
  \rangle\right).
\end{eqnarray*}
By means of this estimate we finally get
\begin{equation}
  \label{k32}
  \|D^\varepsilon_{\vartheta_2 \vartheta_T}\| \leq \varepsilon
  \frac{\bar{\phi} \langle \phi \rangle \bar{b}}{e(\vartheta_2 -
  \vartheta_T)} \exp\left( e^{\vartheta_T}\langle \phi
  \rangle \right).
\end{equation}
Now we recall that $\vartheta= \vartheta_0+\delta (\vartheta_0)$ and then $T(\vartheta,\vartheta_0) = \tau(\vartheta_0)$.
At the same time, $T= \tau(\vartheta_0)/2$, which by (\ref{37b}) yields
\[
 \|Q^\varepsilon_{\vartheta\vartheta_0} (t) \| \leq \|Q^\varepsilon_{\vartheta\vartheta_0} (T) \|\leq 2.
\]
We apply this estimate in the last line of (\ref{k26a}) and then obtain
\begin{equation}
 \label{k33}
 \|J^2_\varepsilon(t)\| \leq 2 \|q_{0,\varepsilon} - e(\varrho_0;\cdot)\|.
\end{equation}
Then the proof of the convergence in (\ref{rho}) follows by (\ref{k26a}), (\ref{k30}), (\ref{k32}) and (\ref{k33}).
\section*{Acknowledgment}
The author was financially supported by National Science Centre,
Poland, grant
\\
2017/25/B/ST1/00051, that is cordially acknowledged by him.

\section*{Appendix}

\subsection*{The proof of (\ref{24})}
For $\gamma \in \Gamma$, let $\eta\Subset \gamma$ mean that
$\eta\subset \gamma$ and $\eta \in \Gamma_0$. Then, see \cite[eq.
(4.18)]{Tobi}, a generalization of (\ref{1fa}) that relates a state
$\mu\in \mathcal{P}_{\rm exp}(\Gamma)$ to its correlation function
$k_\mu$ reads
\begin{equation*}
  \int_{\Gamma} \left(\sum_{\eta \Subset \gamma } G(\eta) \right)
  \mu(d \gamma) = \int_{\Gamma_0} G(\eta) k_\mu( \eta) \lambda (
  d\eta),
\end{equation*}
holding for all appropriate $G:\Gamma_0\to \mathds{R}$. Then by
(\ref{3}) and (\ref{17}), for a given $\mu\in \mathcal{P}_{\rm
exp}(\Gamma)$ and $\theta \in \varTheta$, we have
\begin{eqnarray}
\label{A2} {\rm LHS}(\ref{1g}) & = &
\int_{\Gamma}\int_{\mathds{R}^d} b(x) \theta (x)\prod_{y\in \gamma}
\left[1+ t_x (y) + \tau_x (y)
\theta (y) \right] d x \mu(d\gamma) \\[.2cm] \nonumber  & = &
\int_{\mathds{R}^d} b(x) \theta (x)  \left(\int_{\Gamma}\left(
\sum_{\eta \Subset \gamma} e(\hat{\theta}_x; \eta) \right)  \mu (
d\gamma)\right) d x \\[.2cm] \nonumber  & = &
\int_{\mathds{R}^d} b(x) \theta (x)  \left(\int_{\Gamma_0}
e(\hat{\theta}_x; \eta) k_\mu(\eta) \lambda (d \eta)\right) dx
\\[.2cm] \nonumber  & = &
\int_{\Gamma_0} \int_{\mathds{R}^d} b(x) \theta (x) k_\mu(\eta)
\prod_{y\in \eta} \left[ t_x (y) + \tau_x (y) \theta (y)\right] d x
\lambda ( d\eta)  \\[.2cm] \nonumber  & = & \int_{\mathds{R}^d} b(x)
 \left( \int_{\Gamma_0} k_\mu(\eta)\left( \sum_{\xi\subset \eta}
e(t_x;\xi) e(\tau_x ; \eta\setminus \xi) e(\theta; \eta \cup x
\setminus \xi)\right) \lambda (d \eta) \right) \\[.2cm] \nonumber  & = &
\int_{\Gamma_0} \int_{\Gamma_0} \int_{\mathds{R}^d} b(x)
k_\mu(\eta\cup \xi) e(t_x;\xi) e(\tau_x ; \eta) e(\theta; \eta \cup
x) d x \lambda (d\eta) \lambda (d \xi)  \\[.2cm] \nonumber  & = & \int_{\Gamma_0}
\left(   \sum_{x\in \eta} b(x) e(\tau_x ; \eta\setminus x) \right. \\[.2cm] \nonumber  &  \times &  \left.  \left(
\int_{\Gamma_0}e(t_x;\xi) k_\mu (\eta \setminus x \cup \xi) \lambda
( d\xi) \right) \right) e(\theta; \eta) \lambda ( d \eta)  \\[.2cm] \nonumber  & =
&  \int_{\Gamma_0}\left(L^\Delta k_\mu \right)(\eta) e(\theta; \eta)
\lambda ( d \eta),
\end{eqnarray}
where $L^\Delta$ in the last line of (\ref{A2}) is given in
(\ref{24}). Then the proof follows by (\ref{1g}).

\subsection*{The proof of (\ref{50b})}

For $F = KG$, we have that
\[
F(\eta \cup x) = \sum_{\xi \subset \eta} G(\xi) + \sum _{\xi \subset
\eta} G(\xi\cup x) = F(\eta) + \sum _{\xi \subset \eta} G(\xi\cup x)
.
\]
Then
\begin{equation}
  \label{A4}
  {\rm LHS}(\ref{50b}) (\eta) = \int_{\mathds{R}^d} b^\sigma (x)
  e(\tau_x;\eta) \left(\sum_{\zeta \subset \eta} G(\zeta \cup x)
  \right) dx.
\end{equation}
By (\ref{50a}), we also have
\begin{equation}
  \label{A5}
 {\rm RHS}(\ref{50b}) (\eta) = \int_{\mathds{R}^d} b^\sigma (x) U_x
 (\eta) d x,
\end{equation}
where
\begin{eqnarray}
  \label{A6}
U_x (\eta)& = & \sum_{\zeta \subset \eta} \sum_{\xi\subset \zeta}
e(t_x;\xi) e(\tau_x;\zeta \setminus \xi) G(\zeta \setminus \xi\cup
x)\\[.2cm] \nonumber & = & \sum_{\xi \subset \eta} e(t_x;\xi) \sum_{\zeta \subset \eta\setminus \xi}e(\tau_x;\zeta ) G(\zeta \cup
x)\\[.2cm] \nonumber & = & \sum_{\zeta \subset \eta} e(\tau_x;\zeta ) G(\zeta \cup
x)\sum_{\xi \subset \eta\setminus \zeta} e(t_x;\xi)\\[.2cm] \nonumber & =
& \sum_{\zeta \subset \eta} e(\tau_x;\zeta ) e(\tau_x;\eta\setminus
\zeta ) G(\zeta \cup x) = e (\tau_x;\eta ) \sum_{\zeta \subset \eta}
G(\zeta \cup x).
\end{eqnarray}
In the latter line, we used the following, see the last line in
(\ref{24}),
\[
\sum_{\xi \subset \eta} e(t_x;\xi) = e(1 + t_x; \eta) =
e(\tau_x;\eta).
\]
Now we use (\ref{A6}) w (\ref{A5}) and then by (\ref{A4}) conclude
that (\ref{50b}) holds.

\subsection*{The proof of (\ref{57})}
By (\ref{17}) we have
\begin{gather}
 \label{A3}
 {\rm LHS(\ref{56})} = - \int_{\Gamma_0} \Psi_\sigma (\eta) F(\eta) R(\eta) \lambda (d \eta)\\[.2cm] \nonumber + \int_{\Gamma_0}
 \left(\int_{\mathds{R}^d} b^\sigma (x) e(\tau_x;\eta) F(\eta\cup x) d x  \right) R(\eta) \lambda (d \eta).
\end{gather}
The second line of (\ref{A3}) by (\ref{Minlos}) can be transformed to
\[
 \int_{\Gamma_0}
 \left( \sum_{x\in \eta} b^\sigma (x)
 e(\tau_x;\eta) R(\eta\setminus x)  \right) F(\eta) \lambda (d\eta)= \int_{\Gamma_0}F(\eta) (B R)(\eta) \lambda ( d \eta).
\]
Then by means of (\ref{56}) we conclude that $L^\dag$ acts as described in (\ref{57}).

\subsection*{The proof of (\ref{k29})}

By (\ref{91}) we have that
\begin{equation}
  \label{A7}
 T = \tau(\vartheta_0)/2 = \frac{\delta}{2\bar{b}} \exp\left(\vartheta_0 - \frac{1}{\delta} \right),
\end{equation}
with $\delta>0$ satisfying (\ref{92}). Then
\begin{equation}
  \label{A8}
 e^{\vartheta_T} = e^{\vartheta_0} + \bar{b}T =
 e^{\vartheta_0}\left( 1+ \frac{\delta}{2} \exp\left(-\frac{1}{\delta}
 \right)\right) =: e^{\vartheta_0} v(\delta).
\end{equation}
On the other hand, by (\ref{33}) and then by (\ref{A8}) and
(\ref{A7}) we have
\begin{gather*}
  T(\vartheta, \vartheta_T) = \frac{\vartheta_0 + \delta -
  \vartheta_T}{\bar{b}} \exp\left( \vartheta_T -
  \frac{1}{\delta}\right) \\[.2cm] \nonumber =
  \frac{\delta - \log v(\delta)}{\bar{b}} \exp\left( \vartheta_0 -
  \frac{1}{\delta}\right) v(\delta) = 2T \left(1 - \frac{1}{\delta} \log v(\delta)
  \right) v(\delta).
\end{gather*}
Then (\ref{k29}) turns into the following
\[
2\left(1 - \frac{1}{\delta} \log v(\delta)
  \right) v(\delta) >1,
\]
which is obviously the case as $1<v(\delta)<\exp(\delta/2)$ for each
$\delta>0$.

\end{document}